\documentclass[11pt,reqno,UTF8]{amsart}
\usepackage{amsfonts,verbatim, amscd,mathrsfs, color, array}
\usepackage{amsmath,latexsym,amssymb,amsbsy, amsthm}
\usepackage{graphicx}

\newcommand{\beq}{\begin{equation}}
	\newcommand{\eeq}{\end{equation}}
\newcommand{\ben}{\begin{eqnarray}}
	\newcommand{\een}{\end{eqnarray}}
\newcommand{\beno}{\begin{eqnarray*}}
	\newcommand{\eeno}{\end{eqnarray*}}

\usepackage[hmargin=3.0cm, vmargin=4.0cm]{geometry}
\usepackage{bbm}
\usepackage{pdfsync}
\usepackage{epstopdf}
\usepackage{cite}
\usepackage[english]{babel}
\usepackage{multirow}
\usepackage[font=bf,aboveskip=15pt]{caption}
\usepackage[toc,page]{appendix}
\usepackage[colorlinks,
linkcolor=red,
anchorcolor=blue,
citecolor=blue]{hyperref}
\usepackage{enumerate}
\usepackage{bookmark}
\DeclareFontShape{OT1}{cmr}{bx}{sc}{<-> cmbcsc10}{}
\usepackage{float}
\usepackage{ulem}
\usepackage{syntonly}
\usepackage{mathtools}
\usepackage{bm}
\usepackage{bbding}
\usepackage{wasysym}
\usepackage{diagbox}

\usepackage{upgreek}

\newtheorem{lem}{Lemma}[section]
\newtheorem{prop}{Proposition}[section]
\newtheorem{thm}{Theorem}[section]

\newcommand{\bremark}{\begin{remark} \em}
	\newcommand{\eremark}{\end{remark} }

\allowdisplaybreaks[4]

\numberwithin{equation}{section}
\hypersetup{
	colorlinks=true,
	linkcolor=black, 
	citecolor=black, 
	urlcolor=black 
}
\usepackage{enumitem}
\usepackage{amsthm}

\begin{document}

\title[]{Normalized solutions to subcritical Choquard systems with double couplings}

\author[]{Wenliang Pei and Chonghao Deng}
\address{\noindent
	State Key Laboratory of Mathematical Sciences, Academy of Mathematics and Systems Science, Chinese Academy of Sciences, Beijing 100190, China;
	School of Mathematical Sciences, University of Chinese Academy of Sciences, Beijing 100049, China
}
\email{peiwenliang@amss.ac.cn}

\address{\noindent
	School of Statistics, University of International Business and Economics, Beijing 100029, China
}
\email{dengchonghao@amss.ac.cn}

	\begin{abstract}
	We consider the Choquard system with both linear and nonlinear couplings
\begin{equation*}
	\left\{
	\begin{aligned}
		-\Delta u + \mu_1 u &=\lambda_1 ( I_\alpha * |u|^{r_1} ) |u|^{r_1-2} u + \beta p( I_\alpha * |v|^q)|u|^{p-2} u + \kappa v, \quad \text{in} \,\, \mathbb{R}^N, \\
		-\Delta v + \mu_2 v &=\lambda_2 ( I_\alpha * |v|^{r_2} ) |v|^{r_2-2} v + \beta q( I_\alpha * |u|^p)|v|^{q-2} v + \kappa u , \quad \text{in} \,\, \mathbb{R}^N ,\\
		\int_{\mathbb{R}^N} u^2 = \rho_1^2\, &,  \int_{\mathbb{R}^N} v^2 = \rho_2^2,
	\end{aligned}
	\right.
\end{equation*}
where $N \in \{3,4\}$, $\lambda_1, \lambda_2, \beta, \kappa, \rho_1,\rho_2 > 0$, $2_{\alpha,*} :=\frac{N+\alpha}{N} <p,q , r_1,  r_2 <2_\alpha^*:=\frac{N+\alpha}{N-2}$ and $p+q\leq 2r_1 \leq 2r_2$ .
We investigate a classification result as the parameters $p+q$, $2r_1$ and $2r_2$ vary across the ranges $(\frac{2N+2\alpha}{N},\frac{2N+2\alpha+4}{N})$,
 $\{\frac{2N+2\alpha+4}{N}\}$, and $(\frac{2N+2\alpha+4}{N},\frac{2N+2\alpha}{N-2})$. 
Employing variational methods, we demonstrate the existence of a normalized ground state for the system in the mass subcritical, critical, and supercritical cases.

	\noindent \footnotesize{  \textbf{Keywords: }Normalized solutions;  Choquard systems; linear couplings.} \\
	\footnotesize{ \textbf{AMS Subject Classification (2020): } 35J15; 35J60; 35Q55}
\end{abstract}
\maketitle

	\section{Introduction}

 \par In this paper, we consider the Choquard system with both linear and nonlinear couplings as follows
\begin{equation}
	\left\{
	\begin{aligned}
		-\Delta u + \mu_1 u &=\lambda_1 ( I_\alpha * |u|^{r_1} ) |u|^{r_1-2} u + \beta p( I_\alpha * |v|^q)|u|^{p-2} u + \kappa v, \quad \text{in} \,\, \mathbb{R}^N, \\
		-\Delta v + \mu_2 v &=\lambda_2 ( I_\alpha * |v|^{r_2} ) |v|^{r_2-2} v + \beta q( I_\alpha * |u|^p)|v|^{q-2} v + \kappa u , \quad \text{in} \,\, \mathbb{R}^N ,\\
		\int_{\mathbb{R}^N} u^2 = \rho_1^2\, &,  \int_{\mathbb{R}^N} v^2 = \rho_2^2,
	\end{aligned}
	\right.
	\label{system1.1}
\end{equation}
where $N \in \{3,4\}$, $\lambda_1, \lambda_2, \beta, \kappa, \rho_1,\rho_2 > 0$, $2_{\alpha,*} :=\frac{N+\alpha}{N} <p,q , r_1,  r_2 <2_\alpha^*:=\frac{N+\alpha}{N-2}$ 
and $p+q\leq 2r_1 \leq 2r_2$.
The Riesz potential $I_\alpha: \mathbb{R}^N \to \mathbb{R}$ , with $\alpha \in (0,N)$, takes the form
\begin{equation*}
	I_\alpha(x)=\frac{\Gamma(\frac{N-\alpha}{2})}{2^\alpha \pi^{\frac{N}{2}}\Gamma(\frac{\alpha}{2})}\cdot \frac{1}{|x|^{N-\alpha}},
	 \qquad \forall x \in \mathbb{R}^N\setminus \{0\}.
\end{equation*}

The system \eqref{system1.1} is derived from the subsequent Schr\"{o}dinger system:
\begin{equation}
	\left\{
	\begin{aligned}
		-i\partial_t \Phi_1-\kappa \Phi_2&=\Delta \Phi_1+(H(x)*|\Phi_1|^a)|\Phi_1|^{a-2}\Phi_1+\beta p(H(x)*|\Phi_2|^q)|\Phi_1|^{p-2}\Phi_1,\\
		-i\partial_t \Phi_2-\kappa \Phi_1&=\Delta \Phi_2+(H(x)*|\Phi_2|^b)|\Phi_2|^{b-2}\Phi_2+\beta q(H(x)*|\Phi_1|^p)|\Phi_2|^{q-2}\Phi_2,
	\end{aligned}
	\right.
	\label{system1.2}
\end{equation}
a model stemming from Bose-Einstein condensations \cite{Anderson1995,Deconinck2004}. The Bose–Einstein condensations are of substantial importance in contemporary mathematical physics, leading to a prolific body of research over the last two decades.

Considering $\beta=0$ and $\kappa=0$, the system \eqref{system1.1} simplifies to the subsequent Choquard equation
\begin{equation}
	-\Delta u + \mu_1 u =\lambda_1 ( I_\alpha * |u|^{r_1} ) |u|^{r_1-2} u.
	\label{equation1.3}
\end{equation}
Moroz et al. \cite{MOROZ2013153} carried out important work on this problem and established that equation \eqref{equation1.3} possesses a nontrivial solution 
if and only if $\frac{N+\alpha}{N}<r_1<\frac{N+\alpha}{N-2}$. 

For the situation of $r_1=r_2=\frac{N+\alpha}{N-2}$, Zhang et al. \cite{ZhangandZhong2023} demonstrated the existence of a normalized ground state for the pure
nonlinear couplings case and Pei et al. \cite{Pei2025} demonstrated the existence of a normalized ground state for the double couplings case.
A natural question is whether a normalized ground state exists when $\frac{N+\alpha}{N} <r_1,  r_2 <\frac{N+\alpha}{N-2}$. 

In this paper, we aim to prove the existence of a normalized ground state for the system \eqref{system1.1} in the mass subcritical, critical, and supercritical cases.
Denote $H:=H^1_r(\mathbb{R}^N) \times H^1_r(\mathbb{R}^N)$, where $H^1_r(\mathbb{R}^N)=\{u \in H^1(\mathbb{R}^N):u(x)=u(|x|)\}$. 
 The energy function $\mathcal{J}_\beta : H \to \mathbb{R}$ is given by
\begin{equation*}
	\begin{split}
		\mathcal{J}_\beta(u,v)=&\frac{1}{2}\int_{\mathbb{R}^N}(|\nabla u|^2+|\nabla v|^2)-\frac{\lambda_1}{2r_1}\int_{\mathbb{R}^N}(I_\alpha*|u|^{r_1})|u|^{r_1}-\frac{\lambda_2}{2r_2}\int_{\mathbb{R}^N}(I_\alpha*|v|^{r_2})|v|^{r_2}\\
		&-\beta \int_{\mathbb{R}^N}(I_\alpha*|u|^p)|v|^q-\kappa \int_{\mathbb{R}^N}uv.
	\end{split}
\end{equation*}
Define the $L^2-torus$
\[ \mathcal{S}(\rho_1,\rho_2):= \{(u,v)\in H:\int_{\mathbb{R}^N} u^2 = \rho_1^2 ,~ \int_{\mathbb{R}^N} v^2 = \rho_2^2 \}. \]
It follows that the critical points of $\mathcal{J}_\beta(u,v)$ constrained to  $\mathcal{S}(\rho_1,\rho_2)$ yield the weak solutions of system $\eqref{system1.1} $, 
where the parameters $\mu_1, \mu_2 \in \mathbb{R}$ are the corresponding Lagrangian multipliers. 
Then we state the Pohozaev identity corresponding to system \eqref{system1.1}
\begin{equation}
	\begin{split}
		P_\beta(u,v)=&\int_{\mathbb{R}^N}(|\nabla u|^2+|\nabla v|^2)-\frac{\gamma_{r_1}}{r_1}\lambda_1 \int_{\mathbb{R}^N}(I_\alpha*|u|^{r_1})|u|^{r_1}
		-\frac{\gamma_{r_2}}{r_2}\lambda_2 \int_{\mathbb{R}^N}(I_\alpha*|v|^{r_2})|v|^{r_2}\\
		&-\beta(\gamma_p+\gamma_q)\int_{\mathbb{R}^N}(I_\alpha*|u|^p)|v|^q,
	\end{split}
\end{equation}
where
\begin{equation}
	\gamma_p:=\frac{Np-N-\alpha}{2}.           \label{gamma_p}
\end{equation}
Define $\mathcal{P}_\beta(\rho_1,\rho_2):=\{(u,v)\in  \mathcal{S}(\rho_1,\rho_2): P_\beta(u,v)=0\}$, and we are concerned with the following minimizing problem
\[ m_\beta(\rho_1,\rho_2):= \inf_{ (u,v) \in \mathcal{P}_\beta(\rho_1,\rho_2)}\mathcal{J}_\beta(u,v).
\]
One says that $(u,v)$ is a normalized ground state of system \eqref{system1.1} whenever it realizes $m_\beta(\rho_1,\rho_2)$.

The main results of this paper are presented below. Our first focus is on the $r_1=r_2$ case, for which we have the subsequent theorems.
\begin{thm}
	Assume $N \in \{3,4\}$, $\lambda_1, \lambda_2, \beta, \kappa, \rho_1,\rho_2 > 0$ and $\frac{2N+2\alpha}{\alpha}<p+q\leq 2r_1=2r_2<\frac{2N+2\alpha+4}{N}$, then
	 system $\eqref{system1.1} $ admits a normalized ground state $(u,v)$. Moreover,
	\begin{enumerate}[label=(\roman*)]
		\item $(u,v) \in H_r^1(\mathbb{R}^N) \times H_r^1(\mathbb{R}^N)$ and $u,v> 0$ with the associated Lagrange multiplier $\mu_i$ being positive, $i=1,2$.
		\item  $(u,v)$ is a local minimizer of $\mathcal{J}_\beta$ restricted to $\mathcal{S}(\rho_1,\rho_2)$.
	\end{enumerate}
	\label{Theorem1.1}
\end{thm}

\begin{thm}
	Assume $N \in \{3,4\}$, $\beta, \kappa > 0$, $\frac{2N+2\alpha}{\alpha}<p+q < 2r_1=2r_2=\frac{2N+2\alpha+4}{N}$ and $\rho_1, \rho_2, \lambda_1, \lambda_2>0$ satisfying \eqref{frac{1}{2}-(A1+A2)>0},
	 then system $\eqref{system1.1} $ admits a normalized ground state $(u,v)$. Moreover,
	\begin{enumerate}[label=(\roman*)]
		\item $(u,v) \in H_r^1(\mathbb{R}^N) \times H_r^1(\mathbb{R}^N)$ and $u,v> 0$ with the associated Lagrange multiplier $\mu_i$ being positive, $i=1,2$.
		\item  $(u,v)$ is a local minimizer of $\mathcal{J}_\beta$ restricted to $\mathcal{S}(\rho_1,\rho_2)$.
	\end{enumerate}
	\label{Theorem1.2}
\end{thm}

\begin{thm}
	Assume $N \in \{3,4\}$, $p+q = 2r_1=2r_2=\frac{2N+2\alpha+4}{N}$ and $\rho_1, \rho_2, \lambda_1, \lambda_2, \beta>0$ satisfying 
	{\small
	\begin{equation}
	1-\frac{2N}{N+\alpha+2} C(N,r_1,r_1)2^{\frac{N+\alpha+2}{N}}(\lambda_1\rho_1^{\frac{2\alpha+4}{N}}
	+\lambda_2\rho_2^{\frac{2\alpha+4}{N}})-2\beta C(N,p,q)(\rho_1^2+\rho_2^2)^{\frac{\alpha+2}{N}}>0,
	\label{assumption in Thm 1.3}
	\end{equation}
     }
	then system $\eqref{system1.1} $ admits no normalized solution. 
	\label{Theorem1.3}
\end{thm}

\begin{thm}
	Assume that $N \in \{3,4\}$, $\rho_1, \rho_2, \lambda_1, \lambda_2 > 0$, $\frac{2N+2\alpha}{N}<p+q <\frac{2N+2\alpha+4}{N}< 2r_1=2r_2<\frac{2N+2\alpha}{N-2}$,  $0<\beta<\beta_0$, $0<\kappa<\kappa_0$, where $\beta_0$ and $\kappa_0$ are defined in \eqref{beta 0 in 3.3.1} and \eqref{kappa 0 in 3.3.1} respectively,
	then system $\eqref{system1.1} $ admits a normalized ground state $(u,v)$. Moreover,
	\begin{enumerate}[label=(\roman*)]
		\item $(u,v) \in H_r^1(\mathbb{R}^N) \times H_r^1(\mathbb{R}^N)$ and $u,v> 0$ with the associated Lagrange multiplier $\mu_i$ being positive, $i=1,2$.
		\item  $(u,v)$ is a local minimizer of $\mathcal{J}_\beta$ restricted to $\mathcal{S}(\rho_1,\rho_2)$.
	\end{enumerate}
	\label{Theorem1.4}
\end{thm}

\begin{thm}
Assume $N \in \{3,4\}$, $\lambda_1, \lambda_2, \kappa > 0$, $p+q =\frac{2N+2\alpha+4}{N}< 2r_1=2r_2<\frac{2N+2\alpha}{N-2}$ 
and $\rho_1, \rho_2, \beta>0$ satisfying \eqref{frac{1}{2}-A3>0},
then system $\eqref{system1.1} $ admits a normalized ground state $(u,v)$. Moreover,
\begin{enumerate}[label=(\roman*)]
	\item $(u,v) \in H_r^1(\mathbb{R}^N) \times H_r^1(\mathbb{R}^N)$ and $u,v> 0$ with the associated Lagrange multiplier $\mu_i$ being positive, $i=1,2$.
	\item  $(u,v)$ constitutes a mountain pass critical point for $\mathcal{J}_\beta$ restricted to $\mathcal{S}(\rho_1,\rho_2)$.
\end{enumerate}
\label{Theorem1.5}
\end{thm}

Our second focus is on the $r_1<r_2$ case, for which we have the subsequent theorems.

\begin{thm}
	Assume $N \in \{3,4\}$, $\lambda_1, \lambda_2, \kappa, \rho_1, \rho_2, \beta>0 > 0$ and $\frac{2N+2\alpha+4}{N}<p+q \leq 2r_1=2r_2<\frac{2N+2\alpha}{N-2}$,
	then system $\eqref{system1.1} $ admits a normalized ground state $(u,v)$. Moreover,
	\begin{enumerate}[label=(\roman*)]
		\item $(u,v) \in H_r^1(\mathbb{R}^N) \times H_r^1(\mathbb{R}^N)$ and $u,v> 0$ with the associated Lagrange multiplier $\mu_i$ being positive, $i=1,2$.
		\item  $(u,v)$ constitutes a mountain pass critical point for $\mathcal{J}_\beta$ restricted to $\mathcal{S}(\rho_1,\rho_2)$.
	\end{enumerate}
	\label{Theorem1.6}
\end{thm}

\begin{thm}
	Assume $N \in \{3,4\}$, $\lambda_1, \lambda_2, \beta, \kappa, \rho_1,\rho_2 > 0$ and $\frac{2N+2\alpha}{N}<p+q \leq 2r_1<2r_2<\frac{2N+2\alpha+4}{N}$, then
	system $\eqref{system1.1} $ admits a normalized ground state $(u,v)$. Moreover,
	\begin{enumerate}[label=(\roman*)]
		\item $(u,v) \in H_r^1(\mathbb{R}^N) \times H_r^1(\mathbb{R}^N)$ and $u,v> 0$ with the associated Lagrange multiplier $\mu_i$ being positive, $i=1,2$.
		\item  $(u,v)$ is a local minimizer of $\mathcal{J}_\beta$ restricted to $\mathcal{S}(\rho_1,\rho_2)$.
	\end{enumerate}
	\label{Theorem1.7}
\end{thm}

\begin{thm}
	Assume $N \in \{3,4\}$, $\lambda_1, \rho_1, \kappa, \beta > 0$, $\frac{2N+2\alpha}{N}<p+q \leq 2r_1<2r_2=\frac{2N+2\alpha+4}{N}$ 
	and $\lambda_2, \rho_2>0$ satisfying \eqref{frac{1}{2}-A2>0},
	then system $\eqref{system1.1} $ admits a normalized ground state $(u,v)$. Moreover,
	\begin{enumerate}[label=(\roman*)]
		\item $(u,v) \in H_r^1(\mathbb{R}^N) \times H_r^1(\mathbb{R}^N)$ and $u,v> 0$ with the associated Lagrange multiplier $\mu_i$ being positive, $i=1,2$.
		\item  $(u,v)$ is a local minimizer of $\mathcal{J}_\beta$ restricted to $\mathcal{S}(\rho_1,\rho_2)$.
	\end{enumerate}
	\label{Theorem1.8}
\end{thm}

\begin{thm}
	Assume that $N \in \{3,4\}$, $\frac{2N+2\alpha}{N}<p+q<2r_1<\frac{2N+2\alpha+4}{N}<2r_2<\frac{2N+2\alpha}{N-2}$, $\lambda_1, \lambda_2, \rho_1,\rho_2 > 0$ satisfying \eqref{f(s0)>0 assumption4.2}, $0<\beta<\beta_0$, $0<\kappa<\kappa_0$, where $\beta_0$ and $\kappa_0$ are defined in \eqref{beta 0 in 4.3.1} and \eqref{kappa 0 in 4.3.1} respectively, 
	then system $\eqref{system1.1} $ admits a normalized ground state $(u,v)$. Moreover,
	\begin{enumerate}[label=(\roman*)]
		\item $(u,v) \in H_r^1(\mathbb{R}^N) \times H_r^1(\mathbb{R}^N)$ and $u,v> 0$ with the associated Lagrange multiplier $\mu_i$ being positive, $i=1,2$.
		\item  $(u,v)$ is a local minimizer of $\mathcal{J}_\beta$ restricted to $\mathcal{S}(\rho_1,\rho_2)$.
	\end{enumerate}
	\label{Theorem1.9}
\end{thm}

\begin{thm}
	Assume that $N \in \{3,4\}$, $\frac{2N+2\alpha}{N}<p+q=2r_1<\frac{2N+2\alpha+4}{N}<2r_2<\frac{2N+2\alpha}{N-2}$, $\lambda_1, \lambda_2, \rho_1,\rho_2 > 0$ satisfying \eqref{g(s0)>0 assumption4.8}, $0<\beta<\beta_0$, $0<\kappa<\kappa_0$, where $\beta_0$ and $\kappa_0$ are defined in \eqref{beta 0 in 4.3.2} and \eqref{kappa 0 in 4.3.2} respectively, 
	then system $\eqref{system1.1} $ admits a normalized ground state $(u,v)$. Moreover,
	\begin{enumerate}[label=(\roman*)]
		\item $(u,v) \in H_r^1(\mathbb{R}^N) \times H_r^1(\mathbb{R}^N)$ and $u,v> 0$ with the associated Lagrange multiplier $\mu_i$ being positive, $i=1,2$.
		\item  $(u,v)$ is a local minimizer of $\mathcal{J}_\beta$ restricted to $\mathcal{S}(\rho_1,\rho_2)$.
	\end{enumerate}
	\label{Theorem1.10}
\end{thm}

\begin{thm}
	Assume that $N \in \{3,4\}$, $\frac{2N+2\alpha}{N}<p+q<2r_1=\frac{2N+2\alpha+4}{N}<2r_2<\frac{2N+2\alpha}{N-2}$, $\lambda_1, \lambda_2, \rho_1,\rho_2 > 0$ satisfying \eqref{frac{1}{2}-A_1>0 4.3.3}, $0<\beta<\beta_0$, $0<\kappa<\kappa_0$, where $\beta_0$ and $\kappa_0$ are defined in \eqref{beta 0 in 4.3.3} and \eqref{kappa 0 in 4.3.3} respectively,
	then system $\eqref{system1.1} $ admits a normalized ground state $(u,v)$. Moreover,
	\begin{enumerate}[label=(\roman*)]
		\item $(u,v) \in H_r^1(\mathbb{R}^N) \times H_r^1(\mathbb{R}^N)$ and $u,v> 0$ with the associated Lagrange multiplier $\mu_i$ being positive, $i=1,2$.
		\item  $(u,v)$ is a local minimizer of $\mathcal{J}_\beta$ restricted to $\mathcal{S}(\rho_1,\rho_2)$.
	\end{enumerate}
	\label{Theorem1.11}
\end{thm}

\begin{thm}
	Assume $N \in \{3,4\}$, $ \lambda_2, \kappa > 0$, $\frac{2N+2\alpha}{N}<p+q=2r_1=\frac{2N+2\alpha+4}{N}<2r_2<\frac{2N+2\alpha}{N-2}$
	and $\rho_1, \rho_2, \lambda_1, \beta>0$ satisfying \eqref{frac{1}{2}-(A_1+A_3)>0 4.3.4},
	then system $\eqref{system1.1} $ admits a normalized ground state $(u,v)$. Moreover,
	\begin{enumerate}[label=(\roman*)]
		\item $(u,v) \in H_r^1(\mathbb{R}^N) \times H_r^1(\mathbb{R}^N)$ and $u,v> 0$ with the associated Lagrange multiplier $\mu_i$ being positive, $i=1,2$.
		\item  $(u,v)$ constitutes a mountain pass critical point for $\mathcal{J}_\beta$ restricted to $\mathcal{S}(\rho_1,\rho_2)$.
	\end{enumerate}
	\label{Theorem1.12}
\end{thm}

\begin{thm}
	Assume that $N \in \{3,4\}$, $\frac{2N+2\alpha}{N}<p+q<\frac{2N+2\alpha+4}{N}<2r_1<2r_2<\frac{2N+2\alpha}{N-2}$, $\lambda_1, \lambda_2, \rho_1,\rho_2 > 0$ satisfying \eqref{g(s_0)>0 4.3.5}, $0<\beta<\beta_0$, $0<\kappa<\kappa_0$, where $\beta_0$ and $\kappa_0$ are defined in \eqref{beta 0 in 4.3.5} and \eqref{kappa 0 in 4.3.5} respectively,
	then system $\eqref{system1.1} $ admits a normalized ground state $(u,v)$. Moreover,
	\begin{enumerate}[label=(\roman*)]
		\item $(u,v) \in H_r^1(\mathbb{R}^N) \times H_r^1(\mathbb{R}^N)$ and $u,v> 0$ with the associated Lagrange multiplier $\mu_i$ being positive, $i=1,2$.
		\item  $(u,v)$ is a local minimizer of $\mathcal{J}_\beta$ restricted to $\mathcal{S}(\rho_1,\rho_2)$.
	\end{enumerate}
	\label{Theorem1.13}
\end{thm}

\begin{thm}
	Assume $N \in \{3,4\}$, $\lambda_1, \lambda_2, \kappa > 0$, $p+q=\frac{2N+2\alpha+4}{N}<2r_1<2r_2<\frac{2N+2\alpha}{N-2}$
	and $\rho_1, \rho_2, \beta>0$ satisfying \eqref{frac{1}{2}-A_3>0 4.3.6},
	then system $\eqref{system1.1} $ admits a normalized ground state $(u,v)$. Moreover,
	\begin{enumerate}[label=(\roman*)]
		\item $(u,v) \in H_r^1(\mathbb{R}^N) \times H_r^1(\mathbb{R}^N)$ and $u,v> 0$ with the associated Lagrange multiplier $\mu_i$ being positive, $i=1,2$.
		\item  $(u,v)$ constitutes a mountain pass critical point for $\mathcal{J}_\beta$ restricted to $\mathcal{S}(\rho_1,\rho_2)$.
	\end{enumerate}
	\label{Theorem1.14}
\end{thm}

\begin{thm}
	Assume $N \in \{3,4\}$, $\lambda_1, \lambda_2, \kappa, \rho_1, \rho_2, \beta>0 > 0$ and $\frac{2N+2\alpha+4}{N}<p+q \leq 2r_1<2r_2<\frac{2N+2\alpha}{N-2}$,
	then system $\eqref{system1.1} $ admits a normalized ground state $(u,v)$. Moreover,
	\begin{enumerate}[label=(\roman*)]
		\item $(u,v) \in H_r^1(\mathbb{R}^N) \times H_r^1(\mathbb{R}^N)$ and $u,v> 0$ with the associated Lagrange multiplier $\mu_i$ being positive, $i=1,2$.
		\item  $(u,v)$ constitutes a mountain pass critical point for $\mathcal{J}_\beta$ restricted to $\mathcal{S}(\rho_1,\rho_2)$.
	\end{enumerate}
	\label{Theorem1.15}
\end{thm}

The remainder of this article is organized as follows.
Section 2 introduces the essential technical preliminaries.
In section 3, we study the $r_1=r_2$ case containing $\frac{N+\alpha}{N}<r_1=r_2<\frac{N+\alpha+2}{N}$, 
$r_1=r_2=\frac{N+\alpha+2}{N}$ and $\frac{N+\alpha+2}{N}<r_1=r_2<\frac{N+\alpha}{N-2}$.
In section 4, we study the $r_1<r_2$ case containing $\frac{N+\alpha}{N}<r_1<r_2<\frac{N+\alpha+2}{N}$,
 $\frac{N+\alpha}{N}<r_1<r_2=\frac{N+\alpha+2}{N}$ and $\frac{N+\alpha+2}{N}<r_2<\frac{N+\alpha}{N-2}$.

	\section{PRELIMINARIES}

To complete the presentation, we provide the classical Hardy–Littlewood–Sobolev inequality in the following proposition.
\begin{prop}\cite[theorem 4.3]{Lieb2001}
	Suppose $\alpha \in (0,N), p \in (1, \frac{N}{\alpha})$ and $u \in L^p(\mathbb{R}^N)$. Then $I_\alpha * u \in L^{\frac{Np}{N-\alpha p}}(\mathbb{R}^N)$ and
	\begin{equation}
		\int_{\mathbb{R}^N} |I_\alpha*u|^{\frac{Np}{N-\alpha p}} \leq C \left(\int_{\mathbb{R}^N} |u|^p\right)^{\frac{N}{N-\alpha p}},
		\label{HLS}
	\end{equation}
	where $C$ depends only on $N$, $\alpha$, and $p$.
\end{prop}

We are now able to give the following estimate for $2_{\alpha,*}  <p<2_\alpha^*$:
\[ \int_{\mathbb{R}^N}(I_\alpha *|u|^p)|u|^p \leq C \left(\int_{\mathbb{R}^N} |u|^{\frac{2Np}{N+\alpha}}\right)^{\frac{N+\alpha}{N}} \leq C' |\nabla u|_2^{2 \gamma_p}|u|_2^{2p-2\gamma_p}, \]
where the above inequality is provided by Hölder’s and Gagliardo-Nirenberg inequalities.\\
Next, we make use of the semigroup property of the Riesz potential, namely $ I_\alpha=I_{\frac{\alpha}{2}}*I_{\frac{\alpha}{2}}$ \cite[theorem 5.9]{Lieb2001}
and obtain
\begin{align*}
	\int_{\mathbb{R}^N}(I_\alpha *|u|^p)|v|^q & \leq  \left(\int_{\mathbb{R}^N}(I_\alpha *|u|^p)|u|^p\right)^{\frac{1}{2}} \left(\int_{\mathbb{R}^N}(I_\alpha *|v|^q)|v|^q \right)^{\frac{1}{2}} \\
	& \leq C |\nabla u|_2^{\gamma_p} |u|_2^{p-\gamma_p} |\nabla v|_2^{\gamma_q}  |v|_2^{q-\gamma_q} \\
	& \leq C' \left(|\nabla u|_2^2+|\nabla v|_2^2\right)^{\frac{\gamma_p+\gamma_q}{2}} \left(|u|_2^2+|v|_2^2\right)^{\frac{p+q-(\gamma_p+\gamma_q)}{2}},
\end{align*}
where the last inequality follows from Young’s inequality.\\
Then, defining
\[C(N,p,q)^{-1} := \inf_{(u,v)\in H^1(\mathbb{R}^N)\setminus \{0\}} \frac{\left(|\nabla u|_2^2+|\nabla v|_2^2\right)^{\frac{\gamma_p+\gamma_q}{2}} \left(|u|_2^2+|v|_2^2\right)^{\frac{p+q-\gamma_p-\gamma_q}{2}}}{\int_{\mathbb{R}^N}(I_\alpha *|u|^p)|v|^q},\]
we have the following vector-valued Gagliardo-Nirenberg-type inequality
\begin{equation}
	\int_{\mathbb{R}^N}(I_\alpha *|u|^p)|v|^q \leq C(N,p,q) \left(|\nabla u|_2^2+|\nabla v|_2^2\right)^{\frac{\gamma_p+\gamma_q}{2}} \left(|u|_2^2+|v|_2^2\right)^{\frac{p+q-\gamma_p-\gamma_q}{2}}.
	\label{GN}
\end{equation}

It follows from \eqref{GN} that
\begin{equation*}
	\begin{split}
		\mathcal{J}_\beta(u,v)&=\frac{1}{2}(|\nabla u|_2^2+|\nabla v|_2^2)-\frac{\lambda_1}{2r_1}\int_{\mathbb{R}^N}(I_\alpha*|u|^{r_1})|u|^{r_1}-\frac{\lambda_2}{2r_2}\int_{\mathbb{R}^N}(I_\alpha*|v|^{r_2})|v|^{r_2}\\
		&\quad -\beta \int_{\mathbb{R}^N}(I_\alpha*|u|^p)|v|^q-\kappa \int_{\mathbb{R}^N}uv\\
		&\geq \frac{1}{2}(|\nabla u|_2^2+|\nabla v|_2^2)-\beta C(N,p,q)(\rho_1^2+\rho_2^2)^{\frac{p+q-\gamma_p-\gamma_q}{2}}(|\nabla u|_2^2+|\nabla v|_2^2)^{\frac{\gamma_p+\gamma_q}{2}}-\kappa \rho_1\rho_2\\
		&\quad-\frac{\lambda_1}{2r_1}C(N,r_1,r_1)2^{r_1}\rho_1^{2(r_1-\gamma_{r_1})}|\nabla u|_2^{2\gamma_{r_1}}-\frac{\lambda_2}{2r_2}C(N,r_2,r_2)2^{r_2}\rho_2^{2(r_2-\gamma_{r_2})}|\nabla v|_2^{2\gamma_{r_2}}\\
		&> \frac{1}{2}(|\nabla u|_2^2+|\nabla v|_2^2)-\beta C(N,p,q)(\rho_1^2+\rho_2^2)^{\frac{p+q-\gamma_p-\gamma_q}{2}}(|\nabla u|_2^2+|\nabla v|_2^2)^{\frac{\gamma_p+\gamma_q}{2}}-\kappa \rho_1\rho_2\\
		&\quad -\frac{\lambda_1}{2r_1}C(N,r_1,r_1)2^{r_1}\rho_1^{2(r_1-\gamma_{r_1})}(|\nabla u|_2^2+|\nabla v|_2^2)^{\gamma_{r_1}}\\
		&\quad -\frac{\lambda_2}{2r_2}C(N,r_2,r_2)2^{r_2}\rho_2^{2(r_2-\gamma_{r_2})}(|\nabla u|_2^2+|\nabla v|_2^2)^{\gamma_{r_2}}.
	\end{split}
\end{equation*}
We define
\begin{equation}
	h(s):=\frac{1}{2}s^2-A_1s^{2\gamma_{r_1}}-A_2 s^{2\gamma_{r_2}}-A_3 s^{\gamma_p+\gamma_q}-\kappa \rho_1 \rho_2,
	\label{h(s)}
\end{equation}
where 
\begin{equation*}
	\begin{split}
		A_1:=&\frac{\lambda_1}{2r_1}C(N,r_1,r_1)2^{r_1}\rho_1^{2(r_1-\gamma_{r_1})},\\
		A_2:=&\frac{\lambda_2}{2r_2}C(N,r_2,r_2)2^{r_2}\rho_2^{2(r_2-\gamma_{r_2})},\\
		A_3:=&\beta C(N,p,q)(\rho_1^2+\rho_2^2)^{\frac{p+q-\gamma_p-\gamma_q}{2}}.
	\end{split}
\end{equation*}

At this point, we recall the $L^2$-invariant scaling. For $t>0$, define $t\diamond u(x):=t^{\frac{N}{2}}u(tx)$.
Now for $(u,v) \in \mathcal{S}(\rho_1,\rho_2)$ , we introduce the map $\Psi_{(u,v)}^\beta : (0,+\infty) \to \mathbb{R}$,
\begin{equation*}
	\begin{split}
		\Psi_{(u,v)}^\beta(t):=&\mathcal{J}_\beta(t\diamond(u,v))=\frac{t^2}{2}(|\nabla u|_2^2+|\nabla v|_2^2)-\frac{t^{2\gamma_{r_1}}}{2r_1}\lambda_1\int_{\mathbb{R}^N}(I_\alpha*|u|^{r_1})|u|^{r_1}\\
		&-\frac{t^{2\gamma_{r_2}}}{2r_2}\lambda_2\int_{\mathbb{R}^N}(I_\alpha*|v|^{r_2})|v|^{r_2}-\beta t^{\gamma_p+\gamma_q}\int_{\mathbb{R}^N}(I_\alpha*|u|^p)|v|^q-\kappa \int_{\mathbb{R}^N}uv,
 	\end{split}
\end{equation*}
where $t\diamond(u,v):=(t \diamond u, t \diamond v)$.
Noticing that $(\Psi_{(u,v)}^\beta)'(t)=\frac{1}{t}P_\beta(t\diamond(u,v))$, we have 
\[\mathcal{P}_\beta(\rho_1,\rho_2)=\{(u,v) \in \mathcal{S}(\rho_1,\rho_2) :(\Psi_{(u,v)}^\beta)'(1)=0\}.\]
Then we introduce the decomposition of $\mathcal{P}_\beta(\rho_1,\rho_2)$ as follows:
\begin{gather*}
	\mathcal{P}_\beta^+(\rho_1,\rho_2)=\{(u,v) \in \mathcal{S}(\rho_1,\rho_2):(\Psi_{(u,v)}^\beta)''(1)>0\},\\
	\mathcal{P}_\beta^0(\rho_1,\rho_2)=\{(u,v) \in \mathcal{S}(\rho_1,\rho_2):(\Psi_{(u,v)}^\beta)''(1)=0\},\\
	\mathcal{P}_\beta^-(\rho_1,\rho_2)=\{(u,v) \in \mathcal{S}(\rho_1,\rho_2):(\Psi_{(u,v)}^\beta)''(1)<0\}.
\end{gather*}

We now recall the Liouville type lemma as follows:
\begin{lem}
	\cite[Lemma A.2]{Ikoma2014} Let $r >0$ for $N=1,2$, while $0<r \leq \frac{N}{N-2}$ for $N \geq 3$. Assume $w\in L^r(\mathbb{R}^N)$ is a smooth, nonnegative function and satisfies $- \Delta w \geq 0$ in $\mathbb{R}^N$. Then $w \equiv 0$.
	\label{Liouville}
\end{lem} 

To ensure the regularity properties, we state the following lemma.
	\begin{lem}
	Assume $(u,v) \in H$ solves the Choquard system as follows
\begin{equation*}
		\left\{
	\begin{aligned}
		-\Delta u + \mu_1 u &=\lambda_1 ( I_\alpha * |u|^{r_1} ) |u|^{r_1-2} u + \beta p( I_\alpha * |v|^q)|u|^{p-2} u + \kappa v, \quad \text{in} \,\, \mathbb{R}^N, \\
		-\Delta v + \mu_2 v &=\lambda_2 ( I_\alpha * |v|^{r_2} ) |v|^{r_2-2} v + \beta q( I_\alpha * |u|^p)|v|^{q-2} v + \kappa u , \quad \text{in} \,\, \mathbb{R}^N ,
	\end{aligned}
	\right.
\end{equation*}
	then $(u,v)$ is smooth.
	\label{smooth}
\end{lem}
\begin{proof}
	The argument follows closely that of \cite[Lemma 3.3]{Pei2025}; hence, we omit it here.
\end{proof}

\section{The $r_1=r_2$ case}

\subsection{$\frac{N+\alpha}{N}<r_1=r_2<\frac{N+\alpha+2}{N}$}

\subsubsection{$\frac{2N+2\alpha}{N}<p+q<2r_1=2r_2<\frac{2N+2\alpha+4}{N}$}\leavevmode\par
Under the current assumption, we have $0<\gamma_p+\gamma_q<2\gamma_{r_1}=2\gamma_{r_2}<2$ and
\begin{equation*}
	h(s)=\frac{1}{2}s^2-(A_1+A_2)s^{2\gamma_{r_1}}-A_3 s^{\gamma_p+\gamma_q}-\kappa \rho_1 \rho_2,
\end{equation*}
where $h(s)$ is given by \eqref{h(s)}.
A direct calculation yields
\begin{equation*}
	h'(s)=s-2\gamma_{r_1}(A_1+A_2)s^{2\gamma_{r_1}-1}-(\gamma_p+\gamma_q)A_3s^{\gamma_p+\gamma_q-1}.
\end{equation*}
Set
\begin{equation*}
	g(s):=2\gamma_{r_1}(A_1+A_2)s^{2\gamma_{r_1}-2}+(\gamma_p+\gamma_q)A_3s^{\gamma_p+\gamma_q-2}.
\end{equation*}
Direct computation shows that $\lim\limits_{s\to 0^+}g(s)=+\infty$, $\lim\limits_{s\to +\infty}g(s)=0$ and $g(s)$ is strictly decreasing on $(0,+\infty)$.
Hence, we can find a unique $s_0>0$ satisfying $g(s_0)=1$.
Consequently, $h(s)$ strictly decreases on $(0,s_0)$ while increases on $(s_0,+\infty)$. Noting that 
$\lim\limits_{s\to 0^+}h(s)=-\kappa \rho_1 \rho_2$ and $\lim\limits_{s\to +\infty}h(s)=+\infty$, there exists 
a unique $s_1>0$ such that $h(s_1)=\kappa\rho_1\rho_2$.

\begin{lem}
	Given $(u,v)\in \mathcal{S}(\rho_1,\rho_2)$, the function $\Psi_{(u,v)}^\beta(t)$ admits exactly one critical point $t_\beta(u,v)>0$. Furthermore,
	\begin{enumerate}[label=(\roman*)]
		\item $\mathcal{P}_\beta(\rho_1,\rho_2)=\mathcal{P}_\beta^+(\rho_1,\rho_2)$ and $\mathcal{P}_\beta(\rho_1,\rho_2)$ forms a submanifold of $H$ with codimension 3.
		\item $t\diamond(u,v) \in \mathcal{P}_\beta(\rho_1,\rho_2)$ exactly when $t=t_\beta(u,v)$.
		\item $\Psi_{(u,v)}^\beta(t)$ is increasing on $(t_\beta(u,v),+\infty)$ and 
		\[\Psi_{(u,v)}^\beta(t_\beta(u,v))=\min\limits_{t>0}\Psi_{(u,v)}^\beta(t)=
		\min\{\Psi_{(u,v)}^\beta(t):t(|\nabla u|_2^2+|\nabla v|_2^2)^{\frac{1}{2}}\leq s_1\}.\]
		\item The map $(u,v)\mapsto t_\beta(u,v)$ is of class $C^1$.
	\end{enumerate}
	\label{Lemma3.1变分结构}
\end{lem}
\begin{proof}
$(\romannumeral1)$ Assume, for the sake of contradiction, that  $(u,v)\in \mathcal{P}_\beta(\rho_1,\rho_2)\setminus \mathcal{P}_\beta^+(\rho_1,\rho_2)$. It follows that
\begin{align*}
	(\Psi_{(u,v)}^\beta)'(1)=&|\nabla u|_2^2+|\nabla v|_2^2-\frac{\gamma_{r_1}}{r_1}\lambda_1 \int_{\mathbb{R}^N}(I_\alpha*|u|^{r_1})|u|^{r_1}
	-\frac{\gamma_{r_2}}{r_2}\lambda_2 \int_{\mathbb{R}^N}(I_\alpha*|v|^{r_2})|v|^{r_2}\\
	&-\beta(\gamma_p+\gamma_q)\int_{\mathbb{R}^N}(I_\alpha*|u|^p)|v|^q=0,
\end{align*}
\begin{align*}
	(&\Psi_{(u,v)}^\beta)''(1)=|\nabla u|_2^2+|\nabla v|_2^2
	-\beta(\gamma_p+\gamma_q)(\gamma_p+\gamma_q-1)\int_{\mathbb{R}^N}(I_\alpha*|u|^p)|v|^q\\
	&-\frac{\gamma_{r_1}}{r_1}\lambda_1 (2\gamma_{r_1}-1)\int_{\mathbb{R}^N}(I_\alpha*|u|^{r_1})|u|^{r_1}
	-\frac{\gamma_{r_2}}{r_2}\lambda_2(2\gamma_{r_2}-1)
	\int_{\mathbb{R}^N}(I_\alpha*|v|^{r_2})|v|^{r_2}\leq0.
\end{align*}
By combining the preceding two equations, we derive
\begin{align*}
	\frac{\gamma_{r_1}}{r_1}\lambda_1 (2-2\gamma_{r_1})&\int_{\mathbb{R}^N}(I_\alpha*|u|^{r_1})|u|^{r_1}
	+\frac{\gamma_{r_2}}{r_2}\lambda_2(2-2\gamma_{r_2})\int_{\mathbb{R}^N}(I_\alpha*|v|^{r_2})|v|^{r_2}\\
	&+\beta(\gamma_p+\gamma_q)(2-\gamma_p-\gamma_q)\int_{\mathbb{R}^N}(I_\alpha*|u|^p)|v|^q\leq 0.
\end{align*}
Since $0<\gamma_p+\gamma_q<2\gamma_{r_1}=2\gamma_{r_2}<2$, we obtain $u,v\equiv 0$, which contradicts $(u,v)\in \mathcal{S}(\rho_1,\rho_2)$.

We next show that $\mathcal{P}_\beta(\rho_1,\rho_2)$ forms a submanifold of $H$ with codimension 3. The proof is analogous to \cite[Lemma 5.3]{Soave2020}. For the reader's convenience, we provide the details. Notice that
\begin{equation*}
	\mathcal{P}_\beta(\rho_1,\rho_2)=\{(u,v)\in H: P_\beta(u,v)=0,\, L_1(u)=0,\, L_2(v)=0\},
\end{equation*}
where $L_1(u):=\rho_1^2-|u|_2^2$, $L_2(v):=\rho_2^2-|v|_2^2$. Thus, it suffices to show that $d(P_\beta,L_1,L_2): H \to \mathbb{R}^3$ is surjective, 
for any $(u,v) \in \mathcal{P}_\beta(\rho_1,\rho_2)$. Since $dL_1(u)$ and $dL_2(v)$ are independent, for the sake of contradiction,
 assume that $dP_\beta(u,v)$ can be expressed as a linear combination of $dL_1(u)$, $dL_2(v)$. Hence, there exists $\mu_1, \mu_2 \in \mathbb{R}$ such that $(u,v)$ solves the following system
\begin{equation*}
 	\left\{
 	\begin{aligned}
 		-\Delta u + \mu_1 u &=\lambda_1 \gamma_{r_1}( I_\alpha * |u|^{r_1} ) |u|^{r_1-2} u + \beta p \frac{\gamma_p+\gamma_q}{2}( I_\alpha * |v|^q)|u|^{p-2} u , \quad \text{in} \,\, \mathbb{R}^N, \\
 		-\Delta v + \mu_2 v &=\lambda_2 \gamma_{r_2}( I_\alpha * |v|^{r_2} ) |v|^{r_2-2} v + \beta q\frac{\gamma_p+\gamma_q}{2           }( I_\alpha * |u|^p)|v|^{q-2} v  , \quad \text{in} \,\, \mathbb{R}^N ,\\
 		\int_{\mathbb{R}^N} u^2 = \rho_1^2\, &,  \int_{\mathbb{R}^N} v^2 = \rho_2^2,
 	\end{aligned}
 	\right.
\end{equation*}
By the Pohozaev identity, we obtain
\begin{align*}
|\nabla u|_2^2+|\nabla v|_2^2&-\frac{\gamma_{r_1}^2}{r_1}\lambda_1 \int_{\mathbb{R}^N}(I_\alpha*|u|^{r_1})|u|^{r_1}
	-\frac{\gamma_{r_2}^2}{r_2}\lambda_2 \int_{\mathbb{R}^N}(I_\alpha*|v|^{r_2})|v|^{r_2}\\
	&-\beta\frac{(\gamma_p+\gamma_q)^2}{2}\int_{\mathbb{R}^N}(I_\alpha*|u|^p)|v|^q=0.
\end{align*}
Noting that $P_\beta(u,v)=0$, we derive
\begin{align*}
&|\nabla u|_2^2+|\nabla v|_2^2
	-\beta(\gamma_p+\gamma_q)(\gamma_p+\gamma_q-1)\int_{\mathbb{R}^N}(I_\alpha*|u|^p)|v|^q\\
	&-\frac{\gamma_{r_1}}{r_1}\lambda_1 (2\gamma_{r_1}-1)\int_{\mathbb{R}^N}(I_\alpha*|u|^{r_1})|u|^{r_1}
	-\frac{\gamma_{r_2}}{r_2}\lambda_2(2\gamma_{r_2}-1)
	\int_{\mathbb{R}^N}(I_\alpha*|v|^{r_2})|v|^{r_2}=0,
\end{align*}
which contradicts $\mathcal{P}_\beta^0(\rho_1,\rho_2)=\emptyset$.\\
$(\romannumeral2)$ Recall that
\begin{align*}
P_\beta(t\diamond(u,v))&=t(\Psi_{(u,v)}^\beta)'(t)=t^2(|\nabla u|_2^2+|\nabla v|_2^2)-\frac{t^{2\gamma_{r_1}}\gamma_{r_1}}{r_1}\lambda_1\int_{\mathbb{R}^N}(I_\alpha*|u|^{r_1})|u|^{r_1}\\
&-\frac{t^{2\gamma_{r_2}}\gamma_{r_2}}{r_2}\lambda_2\int_{\mathbb{R}^N}(I_\alpha*|v|^{r_2})|v|^{r_2}-\beta t^{\gamma_p+\gamma_q}(\gamma_p+\gamma_q)\int_{\mathbb{R}^N}(I_\alpha*|u|^p)|v|^q.
\end{align*}
We consider
\begin{align*}
	f(t):=&\frac{t^{2\gamma_{r_1}-2}\gamma_{r_1}}{r_1}\lambda_1\int_{\mathbb{R}^N}(I_\alpha*|u|^{r_1})|u|^{r_1}+
	\frac{t^{2\gamma_{r_2}-2}\gamma_{r_2}}{r_2}\lambda_2\int_{\mathbb{R}^N}(I_\alpha*|v|^{r_2})|v|^{r_2}\\
	&+\beta t^{\gamma_p+\gamma_q-2}(\gamma_p+\gamma_q)\int_{\mathbb{R}^N}(I_\alpha*|u|^p)|v|^q.
\end{align*}
Direct computation shows that $\lim\limits_{s\to 0^+}f(s)=+\infty$, $\lim\limits_{s\to +\infty}f(s)=0$ and $f(s)$ is strictly decreasing on $(0,+\infty)$.
Then $P_\beta(t\diamond(u,v))$ possesses a unique zero $t_\beta(u,v)>0$, from which the remaining results follow.\\
$(\romannumeral3)$ Note that
\[\Psi_{(u,v)}^\beta(t)=\mathcal{J}_\beta(t\diamond(u,v)) \geq h(t(|\nabla u|_2^2+|\nabla v|_2^2)^{\frac{1}{2}}),\]
and
\[\Psi_{(u,v)}^\beta(t)>\kappa\rho_1\rho_2, \quad \forall t\in (\frac{s_1}{(|\nabla u|_2^2+|\nabla v|_2^2)^{\frac{1}{2}}},+\infty).\]
We also have $\lim\limits_{t\to 0^+}\Psi_{(u,v)}^\beta(t)=-\kappa\int_{\mathbb{R}^N}uv \leq \kappa\rho_1\rho_2$, $\lim\limits_{t\to +\infty}\Psi_{(u,v)}^\beta(t)=+\infty$ and $\Psi_{(u,v)}^\beta(t)$ is strictly decreasing for $0<t\ll 1$. Thus, we conclude that
\[\Psi_{(u,v)}^\beta(t_\beta(u,v))=\min\limits_{t>0}\Psi_{(u,v)}^\beta(t)=
\min\{\Psi_{(u,v)}^\beta(t):t(|\nabla u|_2^2+|\nabla v|_2^2)^{\frac{1}{2}}\leq s_1\}.\]
$(\romannumeral4)$ Define $\eta(t,(u,v))=(\Psi_{(u,v)}^\beta)'(t)$. It follows that $\eta(t_\beta(u,v),(u,v))=0$ and
\[\partial_t \eta(t,(u,v))|_{t=t_\beta(u,v)}=(\Psi_{(u,v)}^\beta)''(t_\beta(u,v))>0.\]
Hence, the map $(u,v)\mapsto t_\beta(u,v)$ is of class $C^1$ by the implicit function theorem.
\end{proof}
Let $s>0$, and define
\begin{equation}
	A_s(\rho_1,\rho_2):=\{(u,v) \in \mathcal{S}(\rho_1,\rho_2):|\nabla u|_2^2+|\nabla v|_2^2<s^2\}.
	\label{As rho1 rho2}
\end{equation}
The subsequent lemma provides the properties of $m_\beta(\rho_1,\rho_2)$ and demonstrates its monotonicity.

\begin{lem}
		Given $(u,v)\in \mathcal{S}(\rho_1,\rho_2)$, $\lambda_1, \lambda_2, \beta, \kappa, \rho_1,\rho_2 > 0$, then
	\begin{enumerate}[label=(\roman*)]
		\item $m_\beta(\rho_1,\rho_2)=\inf\limits_{A_{s_1}(\rho_1,\rho_2)}\mathcal{J}_\beta(u,v)=\inf\limits_{\mathcal{P}_\beta^+(\rho_1,\rho_2)}\mathcal{J}_\beta(u,v)<-\kappa\rho_1\rho_2$. Furthermore, we can find $\delta>0$ satisfying
		\[m_\beta(\rho_1,\rho_2)<\inf\limits_{A_{s_1}(\rho_1,\rho_2) \setminus A_{s_1-\delta}(\rho_1,\rho_2)}\mathcal{J}_\beta.\]
		\item $m_\beta(\rho_1,\rho_2) \leq m_\beta(\rho_1',\rho_2')$, for any $0<\rho_1'\leq \rho_1$, $0<\rho_2' \leq \rho_2$.
	\end{enumerate}
	\label{Lemma3.2能量值单调性}
\end{lem}
\begin{proof}
$(\romannumeral1)$	It follows from Lemma \ref{Lemma3.1变分结构} that 
	$ m_\beta(\rho_1,\rho_2)= \inf\limits_{ \mathcal{P}_\beta^+(\rho_1,\rho_2)}\mathcal{J}_\beta(u,v).$
On the one hand, $\mathcal{P}_\beta^+(\rho_1,\rho_2) \subset A_{s_1}(\rho_1,\rho_2)$, then
\[m_\beta(\rho_1,\rho_2) \geq \inf\limits_{A_{s_1}(\rho_1,\rho_2)}\mathcal{J}_\beta(u,v).\]
On the other hand, given $(u,v) \in A_{s_1}(\rho_1,\rho_2)$, since $m_\beta(\rho_1,\rho_2)\leq \mathcal{J}_\beta(t_\beta(u,v)\diamond(u,v)) \leq \mathcal{J}_\beta(u,v)$, we obtain
\[m_\beta(\rho_1,\rho_2) \leq \inf\limits_{A_{s_1}(\rho_1,\rho_2)}\mathcal{J}_\beta(u,v).\]
Let $|u|_2=\rho_1$ and $u>0$, $v:=\frac{\rho_2}{|u|_2}u$, there exists $0<t\ll 1$ such that $t\diamond(u,v) \in A_{s_1}(\rho_1,\rho_2)$. Hence, 
\begin{align*}
	m_\beta&(\rho_1,\rho_2) \leq \mathcal{J}_\beta(t\diamond(u,v))=\frac{t^2}{2}(|\nabla u|^2_2+\frac{\rho_2^2}{\rho_1^2}|\nabla u|^2_2)
	-\frac{\lambda_1 t^{2\gamma_{r_1}}}{2r_1}\int_{\mathbb{R}^N}(I_\alpha*|u|^{r_1})|u|^{r_1}\\
	&-\frac{\lambda_2 t^{2\gamma_{r_2}}}{2r_2}\frac{\rho_2^{2r_2}}{\rho_1^{2r_2}}\int_{\mathbb{R}^N}(I_\alpha*|u|^{r_2})|u|^{r_2}
	-\beta t^{\gamma_p+\gamma_q}(\frac{\rho_2}{\rho_1})^q\int_{\mathbb{R}^N}(I_\alpha*|u|^p)|u|^q-\kappa\rho_1\rho_2\\
	&<-\kappa\rho_1\rho_2, \qquad \text{for } t>0 \text{ small enough}.
\end{align*}
Lastly, we observe that $h(s_1)=\kappa\rho_1\rho_2>0$. There exists $\delta>0$ such that $h(s)\geq \frac{1}{2}m_\beta(\rho_1,\rho_2)$, 
for each $s\in [s_1-\delta,s_1]$. Thus, for $(u,v) \in A_{s_1}(\rho_1,\rho_2) \setminus A_{s_1-\delta}(\rho_1,\rho_2)$, 
we obtain \[\mathcal{J}_\beta(u,v)\geq h\left( (|\nabla u|_2^2+|\nabla v|_2^2)^{\frac{1}{2}} \right)\geq \frac{1}{2}m_\beta(\rho_1,\rho_2)>m_\beta(\rho_1,\rho_2).\]
$(\romannumeral2)$ The argument follows closely that of \cite[Lemma 3.7]{Pei2025}; hence, we omit it here.
\end{proof}

We now proceed to establish the compactness result.
\begin{lem}
	Assume that $N\in \{3,4\}$, $\lambda_1, \lambda_2, \beta, \kappa, \rho_1,\rho_2 > 0$ and  $\{(u_n,v_n)\}\subset \mathcal{S}(\rho_1,\rho_2)$ satisfies that as $n\to \infty$
	\begin{align}
		\mathcal{J}_\beta'(u_n,v_n)+\mu_{1,n}&u_n+\mu_{2,n}v_n\to 0,  \qquad \text{for some}\; \mu_{1,n},\mu_{2,n}\in \mathbb{R}, \label{3.6}\\
		\mathcal{J}_\beta(u_n,v_n)&\to m_\beta(\rho_1,\rho_2), \qquad 	\mathcal{P}_\beta(u_n,v_n)\to 0, \label{3.7}\\
		&u_n^-, v_n^- \to 0 ~~ \text{a.e in}\; \mathbb{R}^N, \label{3.8}
	\end{align}
	with
	\begin{equation}
		c:=m_\beta(\rho_1,\rho_2)<-\kappa\rho_1\rho_2. \label{3.9}
	\end{equation}
	Then there exists $(u,v)\in H$, $u,v>0$ and $\mu_1,\mu_2>0$ satisfying that, for a subsequence,
	$(u_n,v_n)\to (u,v)\; \text{in}\; H$ and $(\mu_{1,n},\mu_{2,n}) \to (\mu_1,\mu_2)\; \text{in} \; \mathbb{R}^2$.
	\label{Lemma3.3紧性}
\end{lem}
\begin{proof}
	We structure the argument in three distinct steps.
	
\text{Step 1:}Recall that
\[\mathcal{J}_\beta(u,v) \geq h((|\nabla u|_2^2+|\nabla v|_2^2)^{\frac{1}{2}}),\]
where $h(s)$ is defined in \eqref{h(s)}. Since $0<\gamma_p+\gamma_q<2\gamma_{r_1}=2\gamma_{r_2}<2$, 
it follows from \eqref{3.7} that $\{(u_n,v_n)\}$ is bounded in $H$.
Drawing from \eqref{3.6}, one can infer that
\[ \mu_{1,n}=-\frac{1}{\rho_1^2}\mathcal{J}_\beta'(u_n,v_n)[(u_n,0)]+o_n(1), \qquad \mu_{2,n}=-\frac{1}{\rho_2^2}\mathcal{J}_\beta'(u_n,v_n)[(0,v_n)]+o_n(1).\]
Therefore, $(\mu_{1,n},\mu_{2,n})$ is bounded in $\mathbb{R}^2$.
Then we can find $(u,v) \in H$, $(\mu_1,\mu_2) \in \mathbb{R}^2$, up to a subsequence, satisfying
\begin{align*}
	&(u_n,v_n)\rightharpoonup (u,v)\quad in  \; H,\\
	&(u_n,v_n)\to (u,v) \quad in  \; L^s(\mathbb{R}^N)\times L^s(\mathbb{R}^N),~s\in (2,2^*),\\
	&(u_n,v_n)\to (u,v) \quad a.e \; in  ~ \mathbb{R}^N,\\
	&(\mu_{1,n},\mu_{2,n}) \to (\mu_1,\mu_2) \quad in  \; \mathbb{R}^2.
\end{align*}
Furthermore, combining \eqref{3.6} and \eqref{3.8} yields that
\begin{align}
	\mathcal{J}_\beta'(u,v)+\mu_1 u+\mu_2 v&=0, \qquad u\geq 0,\; v\geq 0. \label{3.10}\\
	P_\beta(u,v)&=0 .   
	\label{3.11}
\end{align}

\text{Step 2:} Our goal is to demonstrate $u \neq 0$, $v \neq 0$. The maximum principle then ensures $u, v > 0$. 
Assume, for the sake of contradiction, that $u=0$. 
Based on \eqref{3.10}, we obtain
\begin{equation*}
	\left\{
	\begin{aligned}
		0&= \kappa v,\\
		-\Delta v + \mu_2 v &=\lambda_2 ( I_\alpha * |v|^{r_2} ) |v|^{r_2-2} v ,
	\end{aligned}
	\right.
\end{equation*}
Hence, one has $v=0$. Assuming without loss of generality that 
$|\nabla u_n|_2^2 \to a_1 \geq 0$, $|\nabla v_n|_2^2 \to a_2 \geq 0$.
Based on the weak convergence $(u_n,v_n)\rightharpoonup (u,v)$ in $H$, it follows from  \cite[Lemma 2.5]{ZhangandZhong2023} that
\begin{equation}
	\int_{\mathbb{R}^N}(I_\alpha*|u_n|^p)|v_n|^q=\int_{\mathbb{R}^N}(I_\alpha*|u|^p)|v|^q+o_n(1).
	\label{2.12}
\end{equation}
As a result of \eqref{2.12}, one has $\int_{\mathbb{R}^N}(I_\alpha*|u_n|^p)|v_n|^q$, $\int_{\mathbb{R}^N}(I_\alpha*|u_n|^{r_1})|u_n|^{r_1}$ and $\int_{\mathbb{R}^N}(I_\alpha*|v_n|^{r_2})|v_n|^{r_2} \to 0$.
It follows from \eqref{GN} that
\begin{equation*}
\begin{split}
	c+o_n(1)&=\mathcal{J}_\beta(u_n,v_n)=\frac{1}{2}\int_{\mathbb{R}^N}(|\nabla u_n|^2+|\nabla v_n|^2)-\frac{\lambda_1}{2r_1}\int_{\mathbb{R}^N}(I_\alpha*|u_n|^{r_1})|u_n|^{r_1}\\
	&-\frac{\lambda_2}{2r_2}\int_{\mathbb{R}^N}(I_\alpha*|v_n|^{r_2})|v_n|^{r_2}-\beta \int_{\mathbb{R}^N}(I_\alpha*|u_n|^p)|v_n|^q-\kappa \int_{\mathbb{R}^N}u_nv_n\\
	&\geq \frac{1}{2}(a_1+a_2) -\kappa\rho_1\rho_2+o_n(1),
\end{split}
\end{equation*}
which is in contradiction with $c<-\kappa\rho_1\rho_2 $. 

\text{Step 3:} The remainder of this section is devoted to proving that $(u_n,v_n)\to (u,v)$ in $H$.
Assuming, for the sake of contradiction, that $\mu_1 \leq 0$, we derive
\[-\Delta u =- \mu_1 u +\lambda_1 ( I_\alpha * |u|^{r_1} ) |u|^{r_1-2} u + \beta p( I_\alpha * |v|^q)|u|^{p-2} u + \kappa v \geq 0.\]
It follows from Lemma \ref{Liouville} and \ref{smooth} that $u \equiv 0$ in $\mathbb{R}^N$, which contradicts the preceding step.
Hence, $\mu_1 > 0$. By the same argument, $\mu_2 >0$. We define $(\tilde{u}_n,\tilde{v}_n):=(u_n-u,v_n-v)$.
Based on the Brezis-Lieb type lemma \cite[Lemma 3.2]{Yang2013} and \cite[Lemma 2.2]{Chen2021}, the subsequent equalities hold:
\begin{align}
	\int_{\mathbb{R}^N}(I_\alpha*|u_n|^{r_1})|v_n|^{r_1}&=\int_{\mathbb{R}^N}(I_\alpha*|u|^{r_1})|u|^{r_1}+
	\int_{\mathbb{R}^N}(I_\alpha*|\tilde{u}_n|^{r_1})|\tilde{u}_n|^{r_1}+o_n(1)
	\label{Brezis-Lieb of Riesz potential}, \\
	\int_{\mathbb{R}^N}(I_\alpha*|v_n|^{r_2})|v_n|^{r_2}&=\int_{\mathbb{R}^N}(I_\alpha*|v|^{r_2})|v|^{r_2}
	+\int_{\mathbb{R}^N}(I_\alpha*|\tilde{v}_n|^{r_2})|\tilde{v}_n|^{r_2}+o_n(1),\\
	\int_{\mathbb{R}^N}u_nv_n&=\int_{\mathbb{R}^N}uv
	+\int_{\mathbb{R}^N}\tilde{u}_n\tilde{v}_n+o_n(1). \label{Brezis-Lieb of linear couplings}
\end{align}
It follows from \eqref{3.7} and \eqref{3.11} that $P_\beta(\tilde{u}_n,\tilde{v}_n) \to 0$. Since $\int_{\mathbb{R}^N}(I_\alpha*|\tilde{u}_n|^p)|\tilde{v}_n|^q$,
 $\int_{\mathbb{R}^N}(I_\alpha*|\tilde{u}_n|^{r_1})|\tilde{u}_n|^{r_1}$ and $\int_{\mathbb{R}^N}(I_\alpha*|\tilde{v}_n|^{r_2})|\tilde{v}_n|^{r_2} \to 0$, 
 we obtain $|\nabla \tilde{u}_n|^2_2+|\nabla \tilde{v}_n|^2_2 \to 0$.
 
 Lastly, we assume that $|u|_2=\overline{\rho}_1 \in (0,\rho_1]$, $|v|_2=\overline{\rho}_2 \in (0,\rho_2]$.
 Testing \eqref{3.6} and \eqref{3.10} with $(u_n,v_n)$ and $(u,v)$ respectively,
  it follows from \eqref{2.12} and \eqref{Brezis-Lieb of Riesz potential} that
 \[\mu_1(\rho_1^2-\overline{\rho}_1^2)+\mu_2(\rho_2^2-\overline{\rho}_2^2)=
 2\kappa \lim\limits_{n \to +\infty} \int_{\mathbb{R}^N}(u_n-u)(v_n-v).\]
 It is hypothesized that $\overline{\rho}_1=\rho_1$, $\overline{\rho}_2=\rho_2$. Assume, for the sake of contradiction, that $0<\overline{\rho}_1<\rho_1$.
 If $\overline{\rho}_2=\rho_2$, we derive
 \[\mu_1(\rho_1^2-\overline{\rho}_1^2)=2\kappa \lim\limits_{n \to +\infty} \int_{\mathbb{R}^N}(u_n-u)(v_n-v) \leq 2\kappa \sqrt{(\rho_1^2-\overline{\rho}_1^2)(\rho_2^2-\overline{\rho}_2^2)}=0,
 \]
 which is in contradiction with $\overline{\rho}_1<\rho_1$.\\
 If $0<\overline{\rho}_2<\rho_2$, we have
 \begin{align*}
 	2\sqrt{\mu_1\mu_2}\sqrt{(\rho_1^2-\overline{\rho}_1^2)(\rho_2^2-\overline{\rho}_2^2)}&\leq \mu_1(\rho_1^2-\overline{\rho}_1^2)+\mu_2(\rho_2^2-\overline{\rho}_2^2)\\
 	&=2\kappa \lim\limits_{n \to +\infty} \int_{\mathbb{R}^N}(u_n-u)(v_n-v)\\
 	&\leq 2\kappa \sqrt{(\rho_1^2-\overline{\rho}_1^2)(\rho_2^2-\overline{\rho}_2^2)}.
 \end{align*}
 Hence, $\kappa\geq \sqrt{\mu_1\mu_2}$. Assuming that $(\overline{u},\overline{v}):=(\sqrt{\mu_2}u,\sqrt{\mu_1}v)$, we obtain
 \begin{equation*}
 	\left\{
 	\begin{aligned}
 		-\Delta \overline{u}+\mu_1 \overline{u} \geq \mu_2 \overline{v},\\
 		-\Delta \overline{v}+\mu_2 \overline{v} \geq \mu_1 \overline{u}.
 	\end{aligned}
 	\right.
 \end{equation*}
 Thus, $-\Delta(\overline{u}+\overline{v}) \geq 0$. It follows from Lemma \ref{Liouville} that $\overline{u} \equiv 0$, $\overline{v} \equiv 0 $, which contradicts the preceding step.
 Therefore, $(u_n,v_n)\to (u,v)$ in $H$.
\end{proof}
\begin{proof}[\normalfont \textbf{Proof of Theorem 1.1}]\leavevmode\\
We choose $\{(\tilde{u}_n,\tilde{v}_n)\} \subset A_{s_1}(\rho_1,\rho_2)$ as a minimizing sequence associated with $m_\beta(\rho_1,\rho_2)$. 
Note that \[\mathcal{J}_\beta (|\tilde{u}_n|,|\tilde{v}_n|) \leq  \mathcal{J}_\beta(\tilde{u}_n,\tilde{v}_n).\]
Assume without loss of generality that $\tilde{u}_n,\tilde{v}_n \geq 0$. 
Defining $(\overline{u}_n,\overline{v}_n):=t_\beta(\tilde{u}_n,\tilde{v}_n) \diamond (\tilde{u}_n,\tilde{v}_n)$, 
it follows from Lemma \ref{Lemma3.1变分结构} that $(\overline{u}_n,\overline{v}_n) \in \mathcal{P}_\beta(\rho_1,\rho_2)$
and $\mathcal{J}_\beta (\overline{u}_n,\overline{v}_n) \leq \mathcal{J}_\beta (\tilde{u}_n,\tilde{v}_n)$.
Based on Ekeland's variational principle, we can find the Palais-Smale sequence $(u_n,v_n)$ for $\mathcal{J}_\beta|_{\mathcal{S}(\rho_1,\rho_2)}$ 
satisfying
\[||(u_n,v_n)-(\overline{u}_n,\overline{v}_n)||_H \to 0, \qquad \; n \to +\infty. \]
Hence, $P_\beta (u_n,v_n)=P_\beta(\overline{u}_n,\overline{v}_n) + o_n(1)=o_n(1)$ and $u_n^-,v_n^- \to 0$ a.e in $\mathbb{R}^N$.
In the light of Lemma \ref{Lemma3.3紧性}, there exists $(u,v)\in H$, $u,v>0$ and $\mu_1,\mu_2>0$ satisfying that, for a subsequence,
$(u_n,v_n)\to (u,v)\; \text{in}\; H$ and $(\mu_{1,n},\mu_{2,n}) \to (\mu_1,\mu_2)\; \text{in} \; \mathbb{R}^2$. 
Thus, we derive a normalized ground state $(u,v)$ for system \eqref{system1.1}.
\end{proof}

\subsubsection{$\frac{2N+2\alpha}{N}<p+q=2r_1=2r_2<\frac{2N+2\alpha+4}{N}$}\leavevmode\par
Under the current assumption, we have $0<\gamma_p+\gamma_q=2\gamma_{r_1}=2\gamma_{r_2}<2$ and
\begin{equation*}
	h(s)=\frac{1}{2}s^2-(A_1+A_2+A_3)s^{2\gamma_{r_1}}-\kappa \rho_1 \rho_2,
\end{equation*}
where $h(s)$ is given by \eqref{h(s)}. A direct calculation yields
\begin{equation*}
	h'(s)=s-2\gamma_{r_1}(A_1+A_2+A_3)s^{2\gamma_{r_1}-1}.
\end{equation*}
Defining $s_0:=\left(2\gamma_{r_1}(A_1+A_2+A_3)\right)^{\frac{1}{2-2\gamma_{r_1}}}$, we derive
$h(s)$ strictly decreases on $(0,s_0)$ while increases on $(s_0,+\infty)$.
As the proof parallels the preceding case, we conclude the same existence result holds (Theorem \ref{Theorem1.1} when $\frac{2N+2\alpha}{N}<p+q=2r_1=2r_2<\frac{2N+2\alpha+4}{N}$).

\subsection{$r_1=r_2=\frac{N+\alpha+2}{N}$}

\subsubsection{$\frac{2N+2\alpha}{N}<p+q < 2r_1=2r_2=\frac{2N+2\alpha+4}{N}$}\leavevmode\par
Under the current assumption, we have $0<\gamma_p+\gamma_q<2\gamma_{r_1}=2\gamma_{r_2}=2$ and
\begin{align*}
	h(s)&=\frac{1}{2}s^2-(A_1+A_2)s^{2\gamma_{r_1}}-A_3 s^{\gamma_p+\gamma_q}-\kappa \rho_1 \rho_2\\
	&=(\frac{1}{2}-(A_1+A_2))s^2-A_3 s^{\gamma_p+\gamma_q}-\kappa \rho_1 \rho_2,
\end{align*}
where $h(s)$ is given by \eqref{h(s)}. Assume that $\rho_1, \rho_2, \lambda_1, \lambda_2>0$ satisfying 
\begin{equation}
	\frac{1}{2}-(A_1+A_2)=\frac{1}{2}-\frac{\lambda_1}{2r_1}C(N,r_1,r_1)2^{r_1}\rho_1^{2(r_1-\gamma_{r_1})}
	-\frac{\lambda_2}{2r_2}C(N,r_2,r_2)2^{r_2}\rho_2^{2(r_2-\gamma_{r_2})}>0.
	\label{frac{1}{2}-(A1+A2)>0}
\end{equation}
A direct calculation yields
\[    h'(s)=(1-2(A_1+A_2))s-(\gamma_p+\gamma_q)A_3s^{\gamma_p+\gamma_q-1}    .\]
Defining \[s_0=(\frac{A_3(\gamma_p+\gamma_q)}{1-2(A_1+A_2)})^{\frac{1}{2-\gamma_p-\gamma_q}},\]
we derive
$h(s)$ strictly decreases on $(0,s_0)$ while increases on $(s_0,+\infty)$.
As the proof parallels the preceding case, we conclude the same existence result holds (Theorem \ref{Theorem1.2}).

\subsubsection{$p+q = 2r_1=2r_2=\frac{2N+2\alpha+4}{N}$}\leavevmode\par
Under the current assumption, we have $\gamma_p+\gamma_q=2\gamma_{r_1}=2\gamma_{r_2}=2$ and
\begin{align*}
	h(s)&=\frac{1}{2}s^2-(A_1+A_2)s^{2\gamma_{r_1}}-A_3 s^{\gamma_p+\gamma_q}-\kappa \rho_1 \rho_2\\
	&=(\frac{1}{2}-(A_1+A_2+A_3))s^2-\kappa \rho_1 \rho_2.
\end{align*}
The focus now shifts to demonstrating a non-existence result.
\begin{proof}[\normalfont \textbf{Proof of Theorem 1.3}]\leavevmode\\
	Assuming, for the sake of contradiction, that $(u,v)$ solves system \eqref{system1.1}. Then
\begin{equation*}
		\begin{split}
			P_\beta(u,v)=&\int_{\mathbb{R}^N}(|\nabla u|^2+|\nabla v|^2)-\frac{\gamma_{r_1}}{r_1}\lambda_1 \int_{\mathbb{R}^N}(I_\alpha*|u|^{r_1})|u|^{r_1}
			-\frac{\gamma_{r_2}}{r_2}\lambda_2 \int_{\mathbb{R}^N}(I_\alpha*|v|^{r_2})|v|^{r_2}\\
			&-\beta(\gamma_p+\gamma_q)\int_{\mathbb{R}^N}(I_\alpha*|u|^p)|v|^q=0.
		\end{split}
\end{equation*}
It follows from \eqref{GN} that
\begin{align*}
	\int_{\mathbb{R}^N}&(|\nabla u|^2+|\nabla v|^2)=\frac{2}{\frac{N+\alpha+2}{N}}\lambda_1 \int_{\mathbb{R}^N}(I_\alpha*|u|^{r_1})|u|^{r_1}
	+\frac{2}{\frac{N+\alpha+2}{N}}\lambda_2 \int_{\mathbb{R}^N}(I_\alpha*|v|^{r_2})|v|^{r_2}\\
	&\qquad \qquad \qquad \qquad +2\beta\int_{\mathbb{R}^N}(I_\alpha*|u|^p)|v|^q\\
	&\leq \frac{2}{\frac{N+\alpha+2}{N}}\lambda_1 C(N,r_1,r_1)2^{\frac{N+\alpha+2}{N}}\rho_1^{\frac{2\alpha+4}{N}}|\nabla u|_2^2
	+\frac{2}{\frac{N+\alpha+2}{N}}\lambda_2 C(N,r_1,r_1)2^{\frac{N+\alpha+2}{N}}\rho_2^{\frac{2\alpha+4}{N}}|\nabla v|_2^2\\
	&\qquad+2\beta C(N,p,q)(\rho_1^2+\rho_2^2)^{\frac{\alpha+2}{N}}(|\nabla u|_2^2+|\nabla v|_2^2)\\
	&\leq \frac{2N}{N+\alpha+2} C(N,r_1,r_1)2^{\frac{N+\alpha+2}{N}}(\lambda_1\rho_1^{\frac{2\alpha+4}{N}}
	+\lambda_2\rho_2^{\frac{2\alpha+4}{N}})(|\nabla u|_2^2+|\nabla v|_2^2)\\
	&\qquad +2\beta C(N,p,q)(\rho_1^2+\rho_2^2)^{\frac{\alpha+2}{N}}(|\nabla u|_2^2+|\nabla v|_2^2).
\end{align*}
Hence, we have
\begin{equation*}
		1-\frac{2N}{N+\alpha+2} C(N,r_1,r_1)2^{\frac{N+\alpha+2}{N}}(\lambda_1\rho_1^{\frac{2\alpha+4}{N}}
	+\lambda_2\rho_2^{\frac{2\alpha+4}{N}})-2\beta C(N,p,q)(\rho_1^2+\rho_2^2)^{\frac{\alpha+2}{N}} \leq 0,
\end{equation*}
which contradicts \eqref{assumption in Thm 1.3}.
\end{proof}

\subsection{$\frac{N+\alpha+2}{N}<r_1=r_2<\frac{N+\alpha}{N-2}$}
\subsubsection{$\frac{2N+2\alpha}{N}<p+q <\frac{2N+2\alpha+4}{N}< 2r_1=2r_2<\frac{2N+2\alpha}{N-2}$}\leavevmode\par
Under the current assumption, we have $0<\gamma_p+\gamma_q<2<2\gamma_{r_1}=2\gamma_{r_2}<\frac{2N+2\alpha}{N-2}$ and
\begin{equation*}
	h(s)=\frac{1}{2}s^2-(A_1+A_2)s^{2\gamma_{r_1}}-A_3 s^{\gamma_p+\gamma_q}-\kappa \rho_1 \rho_2,
\end{equation*}
where $h(s)$ is given by \eqref{h(s)}. A direct calculation yields
\begin{align*}
	h'(s)&=s-2\gamma_{r_1}(A_1+A_2)s^{2\gamma_{r_1}-1}-(\gamma_p+\gamma_q)A_3s^{\gamma_p+\gamma_q-1}\\
	&=s^{\gamma_p+\gamma_q-1}(s^{2-\gamma_p-\gamma_q}-2\gamma_{r_1}(A_1+A_2)s^{2\gamma_{r_1}-\gamma_p-\gamma_q}-(\gamma_p+\gamma_q)A_3).
\end{align*}
Set
\begin{equation*}
	g(s):=s^{2-\gamma_p-\gamma_q}-2\gamma_{r_1}(A_1+A_2)s^{2\gamma_{r_1}-\gamma_p-\gamma_q}
\end{equation*}
and
\begin{equation}
	s_0=\left(\frac{2-\gamma_p-\gamma_q}{2\gamma_{r_1}(2\gamma_{r_1}-\gamma_p-\gamma_q)(A_1+A_2)}\right)^{\frac{1}{2\gamma_{r_1}-2}}.
	\label{s0 in 3.3.1}
\end{equation}
Direct computation shows that $g(s)$ strictly increases on $(0,s_0)$ while decreases on $(s_0,+\infty)$.
We proceed by defining
\begin{align*}
	\beta_1:=&\frac{s_0^{2-\gamma_p-\gamma_q}-2\gamma_{r_1}(A_1+A_2)s_0^{2\gamma_{r_1}-\gamma_p-\gamma_q}}
	{C(N,p,q)(\gamma_p+\gamma_q)(\rho_1^2+\rho_2^2)^{\frac{p+q-\gamma_p-\gamma_q}{2}}},\\
	\beta_2:=&\frac{\frac{1}{2}s_0^{2-\gamma_p-\gamma_q}-(A_1+A_2)s_0^{2\gamma_{r_1}-\gamma_p-\gamma_q}}
	{C(N,p,q)(\rho_1^2+\rho_2^2)^{\frac{p+q-\gamma_p-\gamma_q}{2}}},
\end{align*}
and
\begin{equation}
	\beta_0:=\min\{\beta_1,\beta_2\}.
	\label{beta 0 in 3.3.1}
\end{equation}
We also denote
\begin{equation}
	\kappa_0:=\frac{\frac{1}{2}s_0^2-(A_1+A_2)s_0^{2\gamma_{r_1}}-\overline{A}_3 s^{\gamma_p+\gamma_q}}{2\rho_1\rho_2},
	\label{kappa 0 in 3.3.1}
\end{equation}
where $\overline{A}_3:=\beta_0 C(N,p,q)(\rho_1^2+\rho_2^2)^{\frac{p+q-\gamma_p-\gamma_q}{2}}.$

Assuming $0<\beta<\beta_0$, $0<\kappa<\kappa_0$, then direct computation shows that $g(s_0)>(\gamma_p+\gamma_q)A_3$,
$h(s_0)>\kappa \rho_1\rho_2>0$. Hence, $h(s)$ possesses precisely two critical points $0<s_1<s_0<s_2$ satisfying
\[h(s_1)=\min\limits_{0<s<s_0} h(s)<0,\quad h(s_2)=\max_{s>0}h(s)>\kappa\rho_1\rho_2.
\]
In addition, we can find $T_1>T_0>0$ such that $h(T_0)=h(T_1)=\kappa \rho_1 \rho_2$ 
and $h(s)>\kappa \rho_1 \rho_2$ exactly when $s \in (T_0,T_1)$.

\begin{lem}
	Assume $0<\beta<\beta_0$, $0<\kappa<\kappa_0$ and $(u,v) \in \mathcal{S}(\rho_1,\rho_2)$, then $\Psi_{(u,v)}^\beta(t)$ has exactly two critical points $s_\beta(u,v)<t_\beta(u,v)$. Furthermore,
	\begin{enumerate}[label=(\roman*)]
		\item $\mathcal{P}_\beta^0(\rho_1,\rho_2)=\emptyset$ and $\mathcal{P}_\beta(\rho_1,\rho_2)$ 
		forms a submanifold of $H$ with codimension 3.
		\item $s_\beta(u,v)\diamond (u,v) \in \mathcal{P}_\beta^+(\rho_1,\rho_2)$, $t_\beta(u,v)\diamond (u,v) \in \mathcal{P}_\beta^-(\rho_1,\rho_2)$ and $t \diamond (u,v) \in \mathcal{P}_\beta(\rho_1,\rho_2)$ 
		if and only if $t=s_\beta(u,v) $ or $ t_\beta(u,v)$.
		\item $\Psi_{(u,v)}^\beta(s_\beta(u,v))=\min\{\Psi_{(u,v)}^\beta(t):t(|\nabla u|_2^2+|\nabla v|_2^2)^{\frac{1}{2}}\leq T_0\}$.
		\item $\Psi_{(u,v)}^\beta(t)$ decreases strictly on $(t_\beta(u,v),+\infty)$ and  \[\Psi_{(u,v)}^\beta(t_\beta(u,v))=\max\limits_{t>0}\Psi_{(u,v)}^\beta(t).\]
		\item The maps $(u,v) \mapsto s_\beta(u,v)$ and $(u,v) \mapsto t_\beta(u,v)$ are of class $C^1$.
	\end{enumerate}
	\label{Lemma3.4变分结构 3.3.1}
\end{lem}
\begin{proof}
$(\romannumeral1)$ Assume, for the sake of contradiction, that  $(u,v)\in \mathcal{P}_\beta^0(\rho_1,\rho_2)$. It follows that
	\begin{align*}
		(\Psi_{(u,v)}^\beta)'(1)=&|\nabla u|_2^2+|\nabla v|_2^2-\frac{\gamma_{r_1}}{r_1}\lambda_1 \int_{\mathbb{R}^N}(I_\alpha*|u|^{r_1})|u|^{r_1}
		-\frac{\gamma_{r_2}}{r_2}\lambda_2 \int_{\mathbb{R}^N}(I_\alpha*|v|^{r_2})|v|^{r_2}\\
		&-\beta(\gamma_p+\gamma_q)\int_{\mathbb{R}^N}(I_\alpha*|u|^p)|v|^q=0,
	\end{align*}
	\begin{align*}
		(&\Psi_{(u,v)}^\beta)''(1)=|\nabla u|_2^2+|\nabla v|_2^2
		-\beta(\gamma_p+\gamma_q)(\gamma_p+\gamma_q-1)\int_{\mathbb{R}^N}(I_\alpha*|u|^p)|v|^q\\
		&-\frac{\gamma_{r_1}}{r_1}\lambda_1 (2\gamma_{r_1}-1)\int_{\mathbb{R}^N}(I_\alpha*|u|^{r_1})|u|^{r_1}
		-\frac{\gamma_{r_2}}{r_2}\lambda_2(2\gamma_{r_2}-1)
		\int_{\mathbb{R}^N}(I_\alpha*|v|^{r_2})|v|^{r_2}=0.
	\end{align*}
	By combining the preceding two equations, we derive
	\begin{align}
		\frac{\gamma_{r_1}}{r_1}\lambda_1 (2-2\gamma_{r_1})&\int_{\mathbb{R}^N}(I_\alpha*|u|^{r_1})|u|^{r_1}
		+\frac{\gamma_{r_2}}{r_2}\lambda_2(2-2\gamma_{r_2})\int_{\mathbb{R}^N}(I_\alpha*|v|^{r_2})|v|^{r_2}\notag\\
		&+\beta(\gamma_p+\gamma_q)(2-\gamma_p-\gamma_q)\int_{\mathbb{R}^N}(I_\alpha*|u|^p)|v|^q=0.
		\label{3.16 3.3.1}
	\end{align}
On the one hand, by \eqref{GN} and \eqref{3.16 3.3.1}, we have
\begin{align}
	|\nabla u|_2^2+|\nabla v|_2^2&=\frac{\gamma_{r_1}}{r_1}\lambda_1 \int_{\mathbb{R}^N}(I_\alpha*|u|^{r_1})|u|^{r_1}
	+\frac{\gamma_{r_2}}{r_2}\lambda_2 \int_{\mathbb{R}^N}(I_\alpha*|v|^{r_2})|v|^{r_2}\notag\\
	&\quad +\beta(\gamma_p+\gamma_q)\int_{\mathbb{R}^N}(I_\alpha*|u|^p)|v|^q \notag  \\
	&=\frac{2\gamma_{r_1}-\gamma_p-\gamma_q}{2\gamma_{r_1}-2}\beta(\gamma_p+\gamma_q)
	\int_{\mathbb{R}^N}(I_\alpha*|u|^p)|v|^q \label{3.17 3.3.1}  \\
	&\leq \frac{2\gamma_{r_1}-\gamma_p-\gamma_q}{2\gamma_{r_1}-2}\beta(\gamma_p+\gamma_q)
	C(N,p,q)(\rho_1^2+\rho_2^2)^{\frac{p+q-\gamma_p-\gamma_q}{2}}(|\nabla u|_2^2+|\nabla v|_2^2)^{\frac{\gamma_p+\gamma_q}{2}}.\notag
\end{align}
On the other hand, by \eqref{GN} and \eqref{3.16 3.3.1}, we obtain
\begin{align}
	|\nabla u|_2^2+|\nabla v|_2^2&=\frac{2\gamma_{r_1}-\gamma_p-\gamma_q}{2-\gamma_p-\gamma_q}\left(
	\frac{\gamma_{r_1}}{r_1}\lambda_1 \int_{\mathbb{R}^N}(I_\alpha*|u|^{r_1})|u|^{r_1}
	+\frac{\gamma_{r_2}}{r_2}\lambda_2 \int_{\mathbb{R}^N}(I_\alpha*|v|^{r_2})|v|^{r_2}\right)\notag\\
	&\leq \frac{2\gamma_{r_1}-\gamma_p-\gamma_q}{2-\gamma_p-\gamma_q}2\gamma_{r_1}(A_1+A_2)
	(	|\nabla u|_2^2+|\nabla v|_2^2)^{\gamma_{r_1}}.     \label{3.18 3.3.1}
\end{align}
By combining \eqref{3.17 3.3.1} and \eqref{3.18 3.3.1}, we get
\[ \beta \geq \frac{(2\gamma_{r_1}-2)s_0^{2-\gamma_p-\gamma_q}}{(2\gamma_{r_1}-\gamma_p-\gamma_q)
	C(N,p,q)(\gamma_p+\gamma_q)(\rho_1^2+\rho_2^2)^{\frac{p+q-\gamma_p-\gamma_q}{2}}}=\beta_1,    \]
which contradicts $\beta<\beta_0$.

Then by an argument analogous to the proof of Lemma \ref{Lemma3.1变分结构} $(\romannumeral1)$, 
we derive $\mathcal{P}_\beta(\rho_1,\rho_2)$ 
forms a submanifold of $H$ with codimension 3; hence, we omit it here.\\
$(\romannumeral2)$-$(\romannumeral4)$ Given $(u,v) \in \mathcal{S}(\rho_1,\rho_2)$, recall that
\[\Psi_{(u,v)}^\beta(t)=\mathcal{J}_\beta(t\diamond(u,v)) \geq h \left( t(	|\nabla u|_2^2+|\nabla v|_2^2)^{\frac{1}{2}} \right),\]
and
\[\Psi_{(u,v)}^\beta(t)> \kappa\rho_1\rho_2,  \qquad \text{for any} ~~ t \in (\frac{T_0}{ (	|\nabla u|_2^2+|\nabla v|_2^2)^{\frac{1}{2}}}, 
\frac{T_1}{ (	|\nabla u|_2^2+|\nabla v|_2^2)^{\frac{1}{2}}}).   \]
We claim that $(\Psi_{(u,v)}^\beta)'(t)=0$ has at most two solutions. A direct calculation yields
\begin{align*}
	&t^{2-\gamma_p-\gamma_q}(|\nabla u|_2^2+|\nabla v|_2^2)-\frac{\gamma_{r_1}}{r_1}\lambda_1 t^{2\gamma_{r_1}-\gamma_p-\gamma_q}
	\int_{\mathbb{R}^N}(I_\alpha*|u|^{r_1})|u|^{r_1}\\
	&\quad -\frac{\gamma_{r_2}}{r_2}\lambda_2 t^{2\gamma_{r_2}-\gamma_p-\gamma_q}\int_{\mathbb{R}^N}(I_\alpha*|v|^{r_2})|v|^{r_2}=
	\beta (\gamma_p+\gamma_q)\int_{\mathbb{R}^N}(I_\alpha*|u|^p)|v|^q,
\end{align*}
where the left-hand side of the equation admits at most one critical point; hence, the claim holds true.\\
Note that $\lim\limits_{t\to 0^+}\Psi_{(u,v)}^\beta(t)=-\kappa\int_{\mathbb{R}^N}uv \leq \kappa\rho_1\rho_2$,  $\lim\limits_{t\to +\infty}\Psi_{(u,v)}^\beta(t)=-\infty$
and $\Psi_{(u,v)}^\beta(t)$ is strictly decreasing for $0<t\ll 1$. Thus, $\Psi_{(u,v)}^\beta(t)$ has exactly two critical points $s_\beta(u,v)<t_\beta(u,v)$ satisfying
\[\Psi_{(u,v)}^\beta(s_\beta(u,v))=\min\{\Psi_{(u,v)}^\beta(t):t(|\nabla u|_2^2+|\nabla v|_2^2)^{\frac{1}{2}}\leq T_0\} \]
and  $\Psi_{(u,v)}^\beta(t_\beta(u,v))=\max\limits_{t>0}\Psi_{(u,v)}^\beta(t)$.\\
$(\romannumeral5)$ By an argument analogous to the proof of Lemma \ref{Lemma3.1变分结构} $(\romannumeral5)$, 
we derive that the maps $(u,v) \mapsto s_\beta(u,v)$ and $(u,v) \mapsto t_\beta(u,v)$ are of class $C^1$; 
hence, we omit it here.
\end{proof}
As the proof parallels the preceding case, we conclude the same existence result holds (Theorem \ref{Theorem1.4}).

\subsubsection{$p+q =\frac{2N+2\alpha+4}{N}< 2r_1=2r_2<\frac{2N+2\alpha}{N-2}$}\leavevmode\par
Under the current assumption, we have $\gamma_p+\gamma_q=2<2\gamma_{r_1}=2\gamma_{r_2}<\frac{2N+2\alpha}{N-2}$ and
\begin{align*}
	h(s)&=\frac{1}{2}s^2-(A_1+A_2)s^{2\gamma_{r_1}}-A_3 s^{\gamma_p+\gamma_q}-\kappa \rho_1 \rho_2\\
	&=(\frac{1}{2}-A_3)s^2-(A_1+A_2)s^{2\gamma_{r_1}}-\kappa \rho_1 \rho_2,
\end{align*}
where $h(s)$ is given by \eqref{h(s)}. Assume that $\beta, \rho_1, \rho_2>0$ satisfying 
\begin{equation}
	\frac{1}{2}-A_3=\frac{1}{2}-\beta C(N,p,q)(\rho_1^2+\rho_2^2)^{\frac{p+q-\gamma_p-\gamma_q}{2}}>0.
	\label{frac{1}{2}-A3>0}
\end{equation}
A direct calculation yields
\[    h'(s)=(1-2A_3)s-2\gamma_{r_1}(A_1+A_2)s^{2\gamma_{r_1}-1}    .\]
Defining \[s_0=(\frac{1-2A_3}{2\gamma_{r_1}(A_1+A_2)})^{\frac{1}{2\gamma_{r_1}-2}},\]
we derive
$h(s)$ strictly increases on $(0,s_0)$ while decreases on $(s_0,+\infty)$.
\begin{lem}
	Given $(u,v)\in \mathcal{S}(\rho_1,\rho_2)$, the function $\Psi_{(u,v)}^\beta(t)$ admits exactly one critical point $t_\beta(u,v)>0$. Furthermore,
	\begin{enumerate}[label=(\roman*)]
		\item $\mathcal{P}_\beta(\rho_1,\rho_2)=\mathcal{P}_\beta^-(\rho_1,\rho_2)$ and $\mathcal{P}_\beta(\rho_1,\rho_2)$ forms a submanifold of $H$ with codimension 3.
		\item $t\diamond(u,v) \in \mathcal{P}_\beta(\rho_1,\rho_2)$ exactly when $t=t_\beta(u,v)$.
		\item $\Psi_{(u,v)}^\beta(t)$ is decreasing on $(t_\beta(u,v),+\infty)$ and 
		\[\Psi_{(u,v)}^\beta(t_\beta(u,v))=\max\limits_{t>0}\Psi_{(u,v)}^\beta(t)>-\kappa \rho_1 \rho_2.\]
		\item The map $(u,v)\mapsto t_\beta(u,v)$ is of class $C^1$.
	\end{enumerate}
	\label{Lemma3.5变分结构 3.3.2}
\end{lem}
\begin{proof}
	By an argument analogous to the proof of Lemma \ref{Lemma3.1变分结构}                 , 
	we derive that the lemma holds; 
	hence, we omit it here.
\end{proof}\noindent
For $\xi>0$, we define
\begin{equation*}
		\Gamma_\xi :=\{\gamma:[0,1]\to \mathcal{S}(\rho_1,\rho_2) : \gamma~ \text{is continuous}, ~\gamma (0)\in A_\xi(\rho_1,\rho_2),~ \mathcal{J}_\beta (\gamma(1)) \leq -\kappa \rho_1 \rho_2\},
\end{equation*}
where $A_\xi(\rho_1,\rho_2)$ is given by \eqref{As rho1 rho2}.
\begin{lem}
	\begin{enumerate}[label=(\roman*)]
		\item There exists $\xi>0$ such that
		\[m_\beta(\rho_1,\rho_2)=\inf\limits_{\mathcal{S}(\rho_1,\rho_2)}\max_{t>0}
		\mathcal{J}_\beta(t\diamond(u,v))=\inf\limits_{\gamma \in \Gamma_\xi}\max\limits_{t\in [0,1]}\mathcal{J}_\beta(\gamma(t))>-\kappa \rho_1\rho_2.	\] 
		\item $m_\beta(\rho_1,\rho_2) \leq m_\beta(\rho_1',\rho_2')$, for every $0<\rho_1'\leq \rho_1$, $0<\rho_2' \leq \rho_2$.
		\item $m_\beta(\rho_1,\rho_2)$ is decreasing in $\beta \in (0,+\infty)$.
	\end{enumerate}
	\label{Lemma3.6能量值 3.3.2}
\end{lem}
\begin{proof}
	$(\romannumeral1)$ It follows from Lemma \ref{Lemma3.5变分结构 3.3.2} that
	\[m_\beta(\rho_1,\rho_2)=\inf\limits_{\mathcal{P}_\beta^-(\rho_1,\rho_2)}
	\mathcal{J}_\beta=\inf\limits_{\mathcal{S}(\rho_1,\rho_2)}\max\limits_{t>0}\mathcal{J}_\beta(t\diamond(u,v)).\]
	We next prove the second equality. On the one hand, assuming $(u,v) \in \mathcal{P}_\beta(\rho_1,\rho_2)$, we define
	$\gamma_M(t):=(tM+\frac{1}{M})\diamond (u,v)$ for $M>0$ large enough. It follows from Lemma \ref{Lemma3.5变分结构 3.3.2} that
	$\gamma_M(t) \in \Gamma_\xi$ and
	\[\mathcal{J}_\beta(u,v)=
	\mathcal{J}_\beta(\gamma_M(\frac{1}{M}-\frac{1}{M^2}))=\max\limits_{t \in [0,1]} \mathcal{J}_\beta(\gamma_M (t))
	 \geq \inf\limits_{\gamma \in \Gamma_\xi}\max_{t\in [0,1]}\mathcal{J}_\beta(\gamma(t)).\] 
	 Hence, $m_\beta(\rho_1,\rho_2) \geq \inf\limits_{\gamma \in \Gamma_\xi}\max\limits_{t\in [0,1]}\mathcal{J}_\beta(\gamma(t)).$\\
	 On the other hand, for $0 < \xi \ll 1$, it follows from Lemma \ref{Lemma3.5变分结构 3.3.2} that
	 $P_\beta(u,v)$ takes positive values on $ A_\xi(\rho_1,\rho_2)$ while $P_\beta(u,v)$ takes negative values on $\mathcal{J}_\beta^{-\kappa\rho_1\rho_2}$,
	 where $\mathcal{J}_\beta^{-\kappa\rho_1\rho_2}:=\{(u,v)\in \mathcal{S}(\rho_1,\rho_2) :\mathcal{J}_\beta(u,v) \leq -\kappa\rho_1\rho_2 \}$. 
	 Thus, for every $\gamma \in \Gamma_\xi$, there exists $t_\gamma \in [0,1]$ such that $P_\beta (\gamma (t_\gamma))=0$. Therefore, we derive
	 \[m_\beta(\rho_1,\rho_2)=\inf\limits_{\mathcal{P}_\beta(\rho_1,\rho_2)}\mathcal{J}_\beta(u,v) \leq \mathcal{J}_\beta(\gamma (t_\gamma)) \leq \max\limits_{t \in [0,1]} \mathcal{J}_\beta(\gamma (t)).\]
	 Hence, $m_\beta(\rho_1,\rho_2) \leq \inf\limits_{\gamma \in \Gamma_\xi}\max\limits_{t\in [0,1]}\mathcal{J}_\beta(\gamma(t)).$
	 
	 We now turn to the proof of the last inequality. Before proceeding, we claim that
	 \begin{equation}
	 	\inf\limits_{\mathcal{P}_\beta(\rho_1,\rho_2)}(|\nabla u|_2^2+|\nabla v|_2^2)>0.
	 	\label{P beta rho1 rho2>0}
	 \end{equation}
	 Indeed, given $(u,v) \in \mathcal{P}_\beta(\rho_1,\rho_2)$, it follows from \eqref{GN} that
\begin{align*}
		|\nabla u|_2^2+|\nabla v|_2^2&=\frac{\gamma_{r_1}}{r_1}\lambda_1 \int_{\mathbb{R}^N}(I_\alpha*|u|^{r_1})|u|^{r_1}
	+\frac{\gamma_{r_2}}{r_2}\lambda_2 \int_{\mathbb{R}^N}(I_\alpha*|v|^{r_2})|v|^{r_2}\\
	&\quad +\beta(\gamma_p+\gamma_q)\int_{\mathbb{R}^N}(I_\alpha*|u|^p)|v|^q   \\
	&\leq 2\gamma_{r_1}A_1(	|\nabla u|_2^2+|\nabla v|_2^2)^{\gamma_{r_1}}+2\gamma_{r_2}A_2(	|\nabla u|_2^2+|\nabla v|_2^2)^{\gamma_{r_2}}\\
	&\quad +(\gamma_p+\gamma_q)A_3(	|\nabla u|_2^2+|\nabla v|_2^2)^{\frac{\gamma_p+\gamma_q}{2}}.
\end{align*}
Since $\gamma_p+\gamma_q=2<2\gamma_{r_1}=2\gamma_{r_2}<\frac{2N+2\alpha}{N-2}$ and $\frac{1}{2}-A_3>0$, the claim holds true.
Then we derive
\begin{align*}
	m_\beta(\rho_1,\rho_2)=&\inf\limits_{(u,v) \in \mathcal{P}_\beta(\rho_1,\rho_2)}\mathcal{J}_\beta(u,v)-\frac{1}{\gamma_p+\gamma_q}P_\beta(u,v)\\
	=&\inf\limits_{(u,v) \in \mathcal{P}_\beta(\rho_1,\rho_2)}\frac{\lambda_1}{2r_1}(\gamma_{r_1}-1)\int_{\mathbb{R}^N}(I_\alpha*|u|^{r_1})|u|^{r_1}\\
	&+\frac{\lambda_2}{2r_2}(\gamma_{r_2}-1)\int_{\mathbb{R}^N}(I_\alpha*|v|^{r_2})|v|^{r_2}-\kappa\int_{\mathbb{R}^N}uv\\
	>&-\kappa\rho_1\rho_2.
\end{align*}
$(\romannumeral2)$ The argument follows closely that of \cite[Lemma 3.7]{Pei2025}; hence, we omit it here.\\
    $(\romannumeral3)$Lastly, by assuming that $0<\beta' \leq \beta$, we obtain
\[m_\beta(\rho_1,\rho_2)=\inf\limits_{\mathcal{S}(\rho_1,\rho_2)}\max_{t>0}
\mathcal{J}_\beta(t\diamond(u,v)) \leq \inf\limits_{\mathcal{S}(\rho_1,\rho_2)}\max_{t>0}
\mathcal{J}_{\beta'}(t\diamond(u,v)) = m_{\beta'}(\rho_1,\rho_2).\]
\end{proof}
	Define $\mathcal{I}_\beta : \mathcal{S}(\rho_1,\rho_2) \to \mathbb{R} \cup \{+\infty\}$ by 
\[\mathcal{I}_\beta (u,v)=\max_{t>0}\mathcal{J}_\beta(t\diamond(u,v))=\max_{t>0} \Psi_{(u,v)}^\beta (t).\]
To establish the existence of a nonnegative Palais-Smale sequence, the subsequent lemmas are essential.
\begin{lem}\cite[Lemma 3.1]{Chen2021Normalized}
	Assume that $\epsilon >0$ and $(u_0,v_0) \in \mathcal{S}(\rho_1,\rho_2)$ satisfies 
	\[\mathcal{I}_\beta(u_0,v_0) \leq \inf\limits_{\mathcal{S}(\rho_1,\rho_2)}\mathcal{I}_\beta(u,v)+\epsilon,\]
	 then given any $\eta >0$,
	there exists $(u_{\eta},v_{\eta}) \in \mathcal{S}(\rho_1,\rho_2) $, $u_{\eta} \geq 0$, $v_{\eta} \geq  0$ satisfying
	\begin{align*}
		&\mathcal{I}_\beta (u_{\eta},v_{\eta}) \leq \mathcal{I}_\beta(u_0,v_0),\\
		&||(u_{\eta},v_{\eta})-(u_0,v_0)||_H \leq \eta,\\
		\mathcal{I}_\beta(u,v) > \mathcal{I}_\beta (u_{\eta},v_{\eta}) - \frac{\epsilon}{\eta}&||(u_{\eta},v_{\eta})-(u,v)||_H , \quad  \forall  (u,v)\in \mathcal{S}(\rho_1,\rho_2) \setminus \{(u_{\eta},v_{\eta})\}.
	\end{align*}
\end{lem}

\begin{lem}
	We can find a nonnegative Palais-Smale sequence $\{(u_n,v_n)\} \subset \mathcal{S}(\rho_1,\rho_2)$, $u_n \geq 0$, $v_n \geq 0$ satisfying
	\begin{align*}
		\mathcal{J}_\beta(u_n,v_n) \to m_\beta(\rho_1,\rho_2), \qquad
		\mathcal{J}_\beta '|_{\mathcal{S}(\rho_1,\rho_2)} (u_n,v_n) \to 0.
	\end{align*}
	Furthermore, the condition can be refined to $\{(u_n,v_n)\} \subset \mathcal{P}_\beta(\rho_1,\rho_2)$. 
	\label{Lemma3.8 PS sequence 3.3.2}
\end{lem}
With the Palais-Smale sequence at hand, we proceed to establish the compactness result.
\begin{lem}
Assume that $N\in \{3,4\}$, $\lambda_1, \lambda_2, \beta, \kappa, \rho_1,\rho_2 > 0$ and $\{(u_n,v_n)\}\subset \mathcal{S}(\rho_1,\rho_2)$ satisfies that as $n\to \infty$
\begin{align}
	\mathcal{J}_\beta'(u_n,v_n)+\mu_{1,n}&u_n+\mu_{2,n}v_n\to 0,  \qquad \text{for some}\; \mu_{1,n},\mu_{2,n}\in \mathbb{R}, \label{3.20 in Lemma3.9}\\
	\mathcal{J}_\beta(u_n,v_n)&\to m_\beta(\rho_1,\rho_2), \qquad 	\mathcal{P}_\beta(u_n,v_n)= 0, \label{3.21 in Lemma3.9}\\
	&u_n, v_n \geq 0 ~~ \text{ in}~ \mathbb{R}^N. \label{3.22 in Lemma3.9}
\end{align}
Then there exists $(u,v)\in H$, $u,v>0$ and $\mu_1,\mu_2>0$ satisfying that, for a subsequence,
$(u_n,v_n)\to (u,v)\; \text{in}\; H$ and $(\mu_{1,n},\mu_{2,n}) \to (\mu_1,\mu_2)\; \text{in} \; \mathbb{R}^2$.
\label{Lemma3.9紧性 3.3.2}
\end{lem}
\begin{proof}
We structure the argument in three distinct steps.
	
\text{Step 1:} It follows from \eqref{3.21 in Lemma3.9} that
\begin{align*}
	m_\beta(\rho_1,\rho_2)+1 &\geq \mathcal{J}_\beta(u_n,v_n)-\frac{1}{\gamma_p+\gamma_q}P_\beta(u_n,v_n)\\
	&=\frac{\lambda_1}{2r_1}(\gamma_{r_1}-1)\int_{\mathbb{R}^N}(I_\alpha*|u_n|^{r_1})|u_n|^{r_1}
	+\frac{\lambda_2}{2r_2}(\gamma_{r_2}-1)\int_{\mathbb{R}^N}(I_\alpha*|v_n|^{r_2})|v_n|^{r_2}\\
	&\quad -\kappa\int_{\mathbb{R}^N}u_n v_n\\
	&\geq \frac{\lambda_1}{2r_1}(\gamma_{r_1}-1)\int_{\mathbb{R}^N}(I_\alpha*|u_n|^{r_1})|u_n|^{r_1}
	+\frac{\lambda_2}{2r_2}(\gamma_{r_2}-1)\int_{\mathbb{R}^N}(I_\alpha*|v_n|^{r_2})|v_n|^{r_2}\\
	&\quad -\kappa \rho_1 \rho_2.
\end{align*}
Since $\gamma_p+\gamma_q=2<2\gamma_{r_1}=2\gamma_{r_2}<\frac{2N+2\alpha}{N-2}$,
we derive that $\int_{\mathbb{R}^N}(I_\alpha*|u_n|^{r_1})|u_n|^{r_1}$, $\int_{\mathbb{R}^N}(I_\alpha*|v_n|^{r_2})|v_n|^{r_2}$
 are bounded; hence, $	|\nabla u_n|_2^2+|\nabla v_n|_2^2-\beta(\gamma_p+\gamma_q)\int_{\mathbb{R}^N}(I_\alpha*|u_n|^p)|v_n|^q $
 is also bounded. A direct calculation yields
 \[ |\nabla u_n|_2^2+|\nabla v_n|_2^2-\beta(\gamma_p+\gamma_q)\int_{\mathbb{R}^N}(I_\alpha*|u_n|^p)|v_n|^q \geq 
 (1-2A_3)(|\nabla u_n|_2^2+|\nabla v_n|_2^2). \]
 Since $\frac{1}{2}-A_3>0$, we obtain that $|\nabla u_n|_2^2+|\nabla v_n|_2^2$ is bounded.
 Drawing from \eqref{3.20 in Lemma3.9}, one can infer that
 \[ \mu_{1,n}=-\frac{1}{\rho_1^2}\mathcal{J}_\beta'(u_n,v_n)[(u_n,0)]+o_n(1), \qquad \mu_{2,n}=-\frac{1}{\rho_2^2}\mathcal{J}_\beta'(u_n,v_n)[(0,v_n)]+o_n(1).\]
 Therefore, $(\mu_{1,n},\mu_{2,n})$ is bounded in $\mathbb{R}^2$.
 Then we can find $(u,v) \in H$, $(\mu_1,\mu_2) \in \mathbb{R}^2$, up to a subsequence, satisfying
 \begin{align*}
 	&(u_n,v_n)\rightharpoonup (u,v)\quad in  \; H,\\
 	&(u_n,v_n)\to (u,v) \quad in  \; L^s(\mathbb{R}^N)\times L^s(\mathbb{R}^N),~s\in (2,2^*),\\
 	&(u_n,v_n)\to (u,v) \quad a.e \; in  ~ \mathbb{R}^N,\\
 	&(\mu_{1,n},\mu_{2,n}) \to (\mu_1,\mu_2) \quad in  \; \mathbb{R}^2.
 \end{align*}
 Furthermore, combining \eqref{3.20 in Lemma3.9} and \eqref{3.22 in Lemma3.9} yields that
 \begin{align}
 	\mathcal{J}_\beta'(u,v)+\mu_1 u+\mu_2 v&=0, \qquad u\geq 0,\; v\geq 0. \label{3.24 in Lemma3.9}\\
 	P_\beta(u,v)&=0 .   
 	\label{3.25 in Lemma3.9}
 \end{align}
 
 \text{Step 2:} Our goal is to demonstrate $u \neq 0$, $v \neq 0$. The maximum principle then ensures $u, v > 0$. 
 Assume, for the sake of contradiction, that $u=0$. 
 Based on \eqref{3.24 in Lemma3.9}, we obtain
 \begin{equation*}
 	\left\{
 	\begin{aligned}
 		0&= \kappa v,\\
 		-\Delta v + \mu_2 v &=\lambda_2 ( I_\alpha * |v|^{r_2} ) |v|^{r_2-2} v ,
 	\end{aligned}
 	\right.
 \end{equation*}
 Hence, one has $v=0$.
 As a result of \eqref{2.12}, one has $\int_{\mathbb{R}^N}(I_\alpha*|u_n|^p)|v_n|^q$, $\int_{\mathbb{R}^N}(I_\alpha*|u_n|^{r_1})|u_n|^{r_1}$ and $\int_{\mathbb{R}^N}(I_\alpha*|v_n|^{r_2})|v_n|^{r_2} \to 0$.
 It follows from \eqref{3.21 in Lemma3.9} that 
 \begin{align*}
 	|\nabla u_n|_2^2+|\nabla v_n|_2^2&=\frac{\gamma_{r_1}}{r_1}\lambda_1 \int_{\mathbb{R}^N}(I_\alpha*|u_n|^{r_1})|u_n|^{r_1}
 	+\frac{\gamma_{r_2}}{r_2}\lambda_2 \int_{\mathbb{R}^N}(I_\alpha*|v_n|^{r_2})|v_n|^{r_2}\\
 	&\quad +\beta(\gamma_p+\gamma_q)\int_{\mathbb{R}^N}(I_\alpha*|u_n|^p)|v_n|^q  \to 0,
 \end{align*}
 which is in contradiction with \eqref{P beta rho1 rho2>0}. 
 
 \text{Step 3:}  By an argument analogous to the step 3 of Lemma \ref{Lemma3.3紧性}, 
 we derive that $(u_n,v_n)\to (u,v)$ in $H$; 
 hence, we omit it here.
\end{proof}
\begin{proof}[\normalfont \textbf{Proof of Theorem 1.5}]\leavevmode\\
	The result is obtained simply by combining Lemma \ref{Lemma3.5变分结构 3.3.2} -Lemma \ref{Lemma3.9紧性 3.3.2}.
\end{proof}

\subsubsection{$\frac{2N+2\alpha+4}{N}<p+q< 2r_1=2r_2<\frac{2N+2\alpha}{N-2}$}\leavevmode\par
Under the current assumption, we have $2<\gamma_p+\gamma_q<2\gamma_{r_1}=2\gamma_{r_2}<\frac{2N+2\alpha}{N-2}$ and
\begin{equation*}
	h(s)=\frac{1}{2}s^2-(A_1+A_2)s^{2\gamma_{r_1}}-A_3 s^{\gamma_p+\gamma_q}-\kappa \rho_1 \rho_2,
\end{equation*}
where $h(s)$ is given by \eqref{h(s)}.
A direct calculation yields
\begin{equation*}
	h'(s)=s-2\gamma_{r_1}(A_1+A_2)s^{2\gamma_{r_1}-1}-(\gamma_p+\gamma_q)A_3s^{\gamma_p+\gamma_q-1}.
\end{equation*}
Set
\begin{equation*}
	g(s):=2\gamma_{r_1}(A_1+A_2)s^{2\gamma_{r_1}-2}+(\gamma_p+\gamma_q)A_3s^{\gamma_p+\gamma_q-2}.
\end{equation*}
Direct computation shows that $\lim\limits_{s\to 0^+}g(s)=0$, $\lim\limits_{s\to +\infty}g(s)=+\infty$ and $g(s)$ is strictly increasing on $(0,+\infty)$.
Hence, we can find a unique $s_0>0$ satisfying $g(s_0)=1$.
Consequently, $h(s)$ strictly increases on $(0,s_0)$ while decreases on $(s_0,+\infty)$.

As the proof parallels the preceding case, we conclude the same existence result holds (Theorem \ref{Theorem1.6} when $\frac{2N+2\alpha+4}{N}<p+q< 2r_1=2r_2<\frac{2N+2\alpha}{N-2}$).

\subsubsection{$\frac{2N+2\alpha+4}{N}<p+q= 2r_1=2r_2<\frac{2N+2\alpha}{N-2}$}\leavevmode\par
Under the current assumption, we have $2<\gamma_p+\gamma_q=2\gamma_{r_1}=2\gamma_{r_2}<\frac{2N+2\alpha}{N-2}$ and
\begin{align*}
	h(s)=&\frac{1}{2}s^2-(A_1+A_2)s^{2\gamma_{r_1}}-A_3 s^{\gamma_p+\gamma_q}-\kappa \rho_1 \rho_2\\
	=&\frac{1}{2}s^2-(A_1+A_2+A_3)s^{2\gamma_{r_1}}-\kappa \rho_1 \rho_2,
\end{align*}
where $h(s)$ is given by \eqref{h(s)}.
A direct calculation yields
\begin{equation*}
	h'(s)=s-2\gamma_{r_1}(A_1+A_2+A_3)s^{2\gamma_{r_1}-1}.
\end{equation*}
Defining \[s_0=(\frac{1}{2\gamma_{r_1}(A_1+A_2+A_3)})^{\frac{1}{2\gamma_{r_1}-2}},\]
we derive
$h(s)$ strictly increases on $(0,s_0)$ while decreases on $(s_0,+\infty)$.

As the proof parallels the preceding case, we conclude the same existence result holds (Theorem \ref{Theorem1.6} when $\frac{2N+2\alpha+4}{N}<p+q= 2r_1=2r_2<\frac{2N+2\alpha}{N-2}$).

\section{The $r_1<r_2$ case}

\subsection{$\frac{N+\alpha}{N}<r_1<r_2<\frac{N+\alpha+2}{N}$}

\subsubsection{$\frac{2N+2\alpha}{N}<p+q<2r_1<2r_2<\frac{2N+2\alpha+4}{N}$}\leavevmode\par
Under the current assumption, we have $0<\gamma_p+\gamma_q<2\gamma_{r_1}<2\gamma_{r_2}<2$ and
\begin{equation*}
	h(s)=\frac{1}{2}s^2-A_1s^{2\gamma_{r_1}}-A_2 s^{2\gamma_{r_2}}-A_3 s^{\gamma_p+\gamma_q}-\kappa \rho_1 \rho_2,
\end{equation*}
where $h(s)$ is given by \eqref{h(s)}.
A direct calculation yields
\begin{equation*}
	h'(s)=s-2\gamma_{r_1}A_1s^{2\gamma_{r_1}-1}-2\gamma_{r_2}A_2s^{2\gamma_{r_2}-1}-(\gamma_p+\gamma_q)A_3s^{\gamma_p+\gamma_q-1}.
\end{equation*}
Set
\begin{equation*}
	g(s):=2\gamma_{r_1}A_1s^{2\gamma_{r_1}-2}+2\gamma_{r_2}A_2s^{2\gamma_{r_2}-2}+(\gamma_p+\gamma_q)A_3s^{\gamma_p+\gamma_q-2}.
\end{equation*}
Direct computation shows that $\lim\limits_{s\to 0^+}g(s)=+\infty$, $\lim\limits_{s\to +\infty}g(s)=0$ and $g(s)$ is strictly decreasing on $(0,+\infty)$.
Hence, we can find a unique $s_0>0$ satisfying $g(s_0)=1$.
Consequently, $h(s)$ strictly decreases on $(0,s_0)$ while increases on $(s_0,+\infty)$. Noting that 
$\lim\limits_{s\to 0^+}h(s)=-\kappa \rho_1 \rho_2$ and $\lim\limits_{s\to +\infty}h(s)=+\infty$, there exists 
a unique $s_1>0$ such that $h(s_1)=\kappa\rho_1\rho_2$.
As the proof parallels the preceding case, we conclude the same existence result holds (Theorem \ref{Theorem1.7} when $\frac{2N+2\alpha}{N}<p+q<2r_1<2r_2<\frac{2N+2\alpha+4}{N}$).

\subsubsection{$\frac{2N+2\alpha}{N}<p+q=2r_1<2r_2<\frac{2N+2\alpha+4}{N}$}\leavevmode\par
Under the current assumption, we have $0<\gamma_p+\gamma_q=2\gamma_{r_1}<2\gamma_{r_2}<2$ and
\begin{align*}
	h(s)&=\frac{1}{2}s^2-A_1s^{2\gamma_{r_1}}-A_2 s^{2\gamma_{r_2}}-A_3 s^{\gamma_p+\gamma_q}-\kappa \rho_1 \rho_2\\
	&=\frac{1}{2}s^2-(A_1+A_3)s^{2\gamma_{r_1}}-A_2 s^{2\gamma_{r_2}}-\kappa \rho_1 \rho_2,
\end{align*}
where $h(s)$ is given by \eqref{h(s)}. A direct calculation yields
\begin{equation*}
	h'(s)=s-2\gamma_{r_1}(A_1+A_3)s^{2\gamma_{r_1}-1}-2\gamma_{r_2}A_2s^{2\gamma_{r_2}-1}.
\end{equation*}
Set
\begin{equation*}
	g(s):=2\gamma_{r_1}(A_1+A_3)s^{2\gamma_{r_1}-2}+2\gamma_{r_2}A_2s^{2\gamma_{r_2}-2}.
\end{equation*}
Direct computation shows that $\lim\limits_{s\to 0^+}g(s)=+\infty$, $\lim\limits_{s\to +\infty}g(s)=0$ and $g(s)$ is strictly decreasing on $(0,+\infty)$.
Hence, we can find a unique $s_0>0$ satisfying $g(s_0)=1$.
Consequently, $h(s)$ strictly decreases on $(0,s_0)$ while increases on $(s_0,+\infty)$. Noting that 
$\lim\limits_{s\to 0^+}h(s)=-\kappa \rho_1 \rho_2$ and $\lim\limits_{s\to +\infty}h(s)=+\infty$, there exists 
a unique $s_1>0$ such that $h(s_1)=\kappa\rho_1\rho_2$.
As the proof parallels the preceding case, we conclude the same existence result holds (Theorem \ref{Theorem1.7} when $\frac{2N+2\alpha}{N}<p+q=2r_1<2r_2<\frac{2N+2\alpha+4}{N}$).

\subsection{$\frac{N+\alpha}{N}<r_1<r_2=\frac{N+\alpha+2}{N}$}

\subsubsection{$\frac{2N+2\alpha}{N}<p+q<2r_1<2r_2=\frac{2N+2\alpha+4}{N}$}\leavevmode\par
Under the current assumption, we have $0<\gamma_p+\gamma_q<2\gamma_{r_1}=2\gamma_{r_2}=2$ and
\begin{align*}
	h(s)&=\frac{1}{2}s^2-A_1s^{2\gamma_{r_1}}-A_2 s^{2\gamma_{r_2}}-A_3 s^{\gamma_p+\gamma_q}-\kappa \rho_1 \rho_2\\
	&=(\frac{1}{2}-A_2)s^2-A_1s^{2\gamma_{r_1}}-A_3 s^{\gamma_p+\gamma_q}-\kappa \rho_1 \rho_2,
\end{align*}
where $h(s)$ is given by \eqref{h(s)}. Assume that $\rho_2,  \lambda_2>0$ satisfying 
\begin{equation}
	\frac{1}{2}-A_2=\frac{1}{2}
	-\frac{\lambda_2}{2r_2}C(N,r_2,r_2)2^{r_2}\rho_2^{2(r_2-\gamma_{r_2})}>0.
	\label{frac{1}{2}-A2>0}
\end{equation}
A direct calculation yields
\[    h'(s)=(1-2A_2)s-2\gamma_{r_1}A_1s^{2\gamma_{r_1}-1}-(\gamma_p+\gamma_q)A_3s^{\gamma_p+\gamma_q-1}    .\]
Set
\begin{equation*}
	g(s):=2\gamma_{r_1}A_1s^{2\gamma_{r_1}-2}+(\gamma_p+\gamma_q)A_3s^{\gamma_p+\gamma_q-2}  .
\end{equation*}
Direct computation shows that $\lim\limits_{s\to 0^+}g(s)=+\infty$, $\lim\limits_{s\to +\infty}g(s)=0$ and $g(s)$ is strictly decreasing on $(0,+\infty)$.
Hence, we can find a unique $s_0>0$ satisfying $g(s_0)=1-2A_2$.
Consequently, $h(s)$ strictly decreases on $(0,s_0)$ while increases on $(s_0,+\infty)$. Noting that 
$\lim\limits_{s\to 0^+}h(s)=-\kappa \rho_1 \rho_2$ and $\lim\limits_{s\to +\infty}h(s)=+\infty$, there exists 
a unique $s_1>0$ such that $h(s_1)=\kappa\rho_1\rho_2$.
As the proof parallels the preceding case, we conclude the same existence result holds (Theorem \ref{Theorem1.8} when $\frac{2N+2\alpha}{N}<p+q<2r_1<2r_2=\frac{2N+2\alpha+4}{N}$).

\subsubsection{$\frac{2N+2\alpha}{N}<p+q=2r_1<2r_2=\frac{2N+2\alpha+4}{N}$}\leavevmode\par
Under the current assumption, we have $\gamma_p+\gamma_q=2\gamma_{r_1}<2\gamma_{r_2}=2$ and
\begin{align*}
	h(s)&=\frac{1}{2}s^2-A_1s^{2\gamma_{r_1}}-A_2 s^{2\gamma_{r_2}}-A_3 s^{\gamma_p+\gamma_q}-\kappa \rho_1 \rho_2\\
	&=(\frac{1}{2}-A_2)s^2-(A_1+A_3)s^{2\gamma_{r_1}}-\kappa \rho_1 \rho_2,
\end{align*}
where $h(s)$ is given by \eqref{h(s)}. 
Assume that $\rho_2,  \lambda_2>0$ satisfying \eqref{frac{1}{2}-A2>0}.
A direct calculation yields
\begin{equation*}
	h'(s)=(1-2A_2)s-2\gamma_{r_1}(A_1+A_3)s^{2\gamma_{r_1}-1}.
\end{equation*}
Defining \[s_0=(\frac{1-2A_3}{2\gamma_{r_1}(A_1+A_2)})^{\frac{1}{2\gamma_{r_1}-2}},\]
we derive
$h(s)$ strictly decreases on $(0,s_0)$ while increases on $(s_0,+\infty)$.
As the proof parallels the preceding case, we conclude the same existence result holds (Theorem \ref{Theorem1.8} when $\frac{2N+2\alpha}{N}<p+q=2r_1<2r_2=\frac{2N+2\alpha+4}{N}$).

\subsection{$\frac{N+\alpha+2}{N}<r_2<\frac{N+\alpha}{N-2}$}

\subsubsection{$\frac{2N+2\alpha}{N}<p+q<2r_1<\frac{2N+2\alpha+4}{N}<2r_2<\frac{2N+2\alpha}{N-2}$}\leavevmode\par
Under the current assumption, we have $0<\gamma_p+\gamma_q<2\gamma_{r_1}<2<2\gamma_{r_2}<\frac{2N+2\alpha}{N-2}$ and
\begin{align*}
	h(s)=\frac{1}{2}s^2-A_1s^{2\gamma_{r_1}}-A_2 s^{2\gamma_{r_2}}-A_3 s^{\gamma_p+\gamma_q}-\kappa \rho_1 \rho_2,
\end{align*}
where $h(s)$ is given by \eqref{h(s)}. 
A direct calculation yields
\begin{align*}
	h'(s)&=s-2\gamma_{r_1}A_1s^{2\gamma_{r_1}-1}-2\gamma_{r_2}A_2s^{2\gamma_{r_2}-1}-(\gamma_p+\gamma_q)A_3s^{\gamma_p+\gamma_q-1}\\
	&=s^{\gamma_p+\gamma_q-1}(s^{2-\gamma_p-\gamma_q}-2\gamma_{r_1}A_1s^{2\gamma_{r_1}-\gamma_p-\gamma_q}
	-2\gamma_{r_2}A_2s^{2\gamma_{r_2}-\gamma_p-\gamma_q}-(\gamma_p+\gamma_q)A_3).
\end{align*}
Set
\begin{equation*}
	g(s):=s^{2-\gamma_p-\gamma_q}-2\gamma_{r_1}A_1s^{2\gamma_{r_1}-\gamma_p-\gamma_q}
	-2\gamma_{r_2}A_2s^{2\gamma_{r_2}-\gamma_p-\gamma_q}  ,
\end{equation*}
and we derive
\begin{align*}
	g'(s)=&(2-\gamma_p-\gamma_q)s^{1-\gamma_p-\gamma_q}-2\gamma_{r_1}(2\gamma_{r_1}-\gamma_p-\gamma_q)A_1s^{2\gamma_{r_1}-\gamma_p-\gamma_q-1}\\
	&-2\gamma_{r_2}(2\gamma_{r_2}-\gamma_p-\gamma_q)A_2s^{2\gamma_{r_2}-\gamma_p-\gamma_q-1} \\
	=&s^{1-\gamma_p-\gamma_q}[(2-\gamma_p-\gamma_q)-2\gamma_{r_1}(2\gamma_{r_1}-\gamma_p-\gamma_q)A_1s^{2\gamma_{r_1}-2}\\
	&-2\gamma_{r_2}(2\gamma_{r_2}-\gamma_p-\gamma_q)A_2s^{2\gamma_{r_2}-2} ].
\end{align*}
Now we define
\begin{align*}
	f(s):=&2\gamma_{r_1}(2\gamma_{r_1}-\gamma_p-\gamma_q)A_1s^{2\gamma_{r_1}-2}+2\gamma_{r_2}(2\gamma_{r_2}-\gamma_p-\gamma_q)A_2s^{2\gamma_{r_2}-2} \\
	=&A_1's^{2\gamma_{r_1}-2}+A_2's^{2\gamma_{r_2}-2},
\end{align*}
where $A_1':=2\gamma_{r_1}(2\gamma_{r_1}-\gamma_p-\gamma_q)A_1$ and $A_2':=2\gamma_{r_2}(2\gamma_{r_2}-\gamma_p-\gamma_q)A_2$.\\
Defining 
\[s_0= \left[ \frac{(1-\gamma_{r_1})A_1'}{(\gamma_{r_2}-1)A_2'} \right] ^{\frac{1}{2\gamma_{r_2}-2\gamma_{r_1}}},\]
we derive
$f(s)$ strictly decreases on $(0,s_0)$ while increases on $(s_0,+\infty)$. Direct computation shows that
\[ f(s_0)=\left[(\frac{\gamma_{r_2}-1}{1-\gamma_{r_1}})^{\frac{1-\gamma_{r_1}}{\gamma_{r_2}-\gamma_{r_1}}}+
 (\frac{1-\gamma_{r_1}}{\gamma_{r_2}-1})^{\frac{\gamma_{r_2}-1}{\gamma_{r_2}-\gamma_{r_1}}}\right]
 (A_1')^{\frac{\gamma_{r_2}-1}{\gamma_{r_2}-\gamma_{r_1}}}(A_2')^{\frac{1-\gamma_{r_1}}{\gamma_{r_2}-\gamma_{r_1}}}.
\]
Assume that $\lambda_1, \lambda_2, \rho_1,\rho_2 > 0$ satisfying 

\begin{equation}
	\left\{
	\begin{aligned}
	f(s_0)<&2-\gamma_p-\gamma_q,\\
g(s_0)>&0,\\
\frac{1}{2}s_0^2-A_1s_0^{2\gamma_{r_1}}-A_2 s_0^{2\gamma_{r_2}}>&0,\\
(\frac{\gamma_{r_2}-\gamma_{r_1}}{\gamma_{r_2}-1}2\gamma_{r_1}A_1)^{\gamma_{r_2}-1}
(&\frac{\gamma_{r_2}-\gamma_{r_1}}{1-\gamma_{r_1}}2\gamma_{r_2}A_2)^{1-\gamma_{r_1}}<1.
	\end{aligned}
	\right.
	\label{f(s0)>0 assumption4.2}
\end{equation}
Direct computation shows that \eqref{f(s0)>0 assumption4.2} holds for sufficiently small $A_1, A_2>0$, 
justifying \eqref{f(s0)>0 assumption4.2}  as a reasonable assumption.
Set
\[ U:=(\frac{1-\gamma_{r_1}}{(\gamma_{r_2}-\gamma_{r_1})2\gamma_{r_2}A_2})^\frac{1}{\gamma_{r_2}-1},
\]
and define $\beta_1>0$ as the solution to the following equation
\begin{align}
	U=&\beta\frac{2\gamma_{r_2}-\gamma_p-\gamma_q}{2\gamma_{r_2}-2}(\gamma_p+\gamma_q)
	C(N,p,q)(\rho_1^2+\rho_2^2)^{\frac{p+q-\gamma_p-\gamma_q}{2}}U^{\frac{\gamma_p+\gamma_q}{2}}\notag \\
	&+\frac{\gamma_{r_2}-\gamma_{r_1}}{\gamma_{r_2}-1}2\gamma_{r_1}A_1 U^{\gamma_{r_1}}
	\label{beta_1 4.3.1}
\end{align}
We proceed by defining
\begin{align*}
	\beta_2:=&\frac{s_0^{2-\gamma_p-\gamma_q}-2\gamma_{r_1}A_1s_0^{2\gamma_{r_1}-\gamma_p-\gamma_q}-
	2\gamma_{r_2}A_2s_0^{2\gamma_{r_2}-\gamma_p-\gamma_q}}
	{C(N,p,q)(\gamma_p+\gamma_q)(\rho_1^2+\rho_2^2)^{\frac{p+q-\gamma_p-\gamma_q}{2}}},\\
	\beta_3:=&\frac{\frac{1}{2}s_0^{2-\gamma_p-\gamma_q}-A_1s_0^{2\gamma_{r_1}-\gamma_p-\gamma_q}
	-A_2s_0^{2\gamma_{r_2}-\gamma_p-\gamma_q}}
	{C(N,p,q)(\rho_1^2+\rho_2^2)^{\frac{p+q-\gamma_p-\gamma_q}{2}}},
\end{align*}
and
\begin{equation}
	\beta_0:=\min\{\beta_1,\beta_2, \beta_3\}.
	\label{beta 0 in 4.3.1}
\end{equation}
We also denote
\begin{equation}
	\kappa_0:=\frac{\frac{1}{2}s_0^2-A_1s_0^{2\gamma_{r_1}}
		-A_2s_0^{2\gamma_{r_2}}-\overline{A}_3 s^{\gamma_p+\gamma_q}}{2\rho_1\rho_2},
	\label{kappa 0 in 4.3.1}
\end{equation}
where $\overline{A}_3:=\beta_0 C(N,p,q)(\rho_1^2+\rho_2^2)^{\frac{p+q-\gamma_p-\gamma_q}{2}}.$

Assuming $0<\beta<\beta_0$, $0<\kappa<\kappa_0$, then direct computation shows that $g(s_0)>(\gamma_p+\gamma_q)A_3$,
$h(s_0)>\kappa \rho_1\rho_2>0$. Hence, $h(s)$ possesses precisely two critical points $0<s_1<s_0<s_2$ satisfying
\[h(s_1)=\min\limits_{0<s<s_0} h(s)<0,\quad h(s_2)=\max_{s>0}h(s)>\kappa\rho_1\rho_2.
\]
In addition, we can find $T_1>T_0>0$ such that $h(T_0)=h(T_1)=\kappa \rho_1 \rho_2$ 
and $h(s)>\kappa \rho_1 \rho_2$ exactly when $s \in (T_0,T_1)$.

\begin{lem}
	Assume that $(u,v) \in \mathcal{S}(\rho_1,\rho_2)$, $\lambda_1, \lambda_2, \rho_1,\rho_2 > 0$ satisfying \eqref{f(s0)>0 assumption4.2}, $0<\beta<\beta_0$, $0<\kappa<\kappa_0$, where $\beta_0$ and $\kappa_0$ are defined in \eqref{beta 0 in 4.3.1} and \eqref{kappa 0 in 4.3.1} respectively, then $\Psi_{(u,v)}^\beta(t)$ has exactly two critical points $s_\beta(u,v)<t_\beta(u,v)$. Furthermore,
	\begin{enumerate}[label=(\roman*)]
		\item $\mathcal{P}_\beta^0(\rho_1,\rho_2)=\emptyset$ and $\mathcal{P}_\beta(\rho_1,\rho_2)$ 
		forms a submanifold of $H$ with codimension 3.
		\item $s_\beta(u,v)\diamond (u,v) \in \mathcal{P}_\beta^+(\rho_1,\rho_2)$, $t_\beta(u,v)\diamond (u,v) \in \mathcal{P}_\beta^-(\rho_1,\rho_2)$ and $t \diamond (u,v) \in \mathcal{P}_\beta(\rho_1,\rho_2)$ 
		if and only if $t=s_\beta(u,v) $ or $ t_\beta(u,v)$.
		\item $\Psi_{(u,v)}^\beta(s_\beta(u,v))=\min\{\Psi_{(u,v)}^\beta(t):t(|\nabla u|_2^2+|\nabla v|_2^2)^{\frac{1}{2}}\leq T_0\}$.
		\item $\Psi_{(u,v)}^\beta(t)$ decreases strictly on $(t_\beta(u,v),+\infty)$ and  \[\Psi_{(u,v)}^\beta(t_\beta(u,v))=\max\limits_{t>0}\Psi_{(u,v)}^\beta(t).\]
		\item The maps $(u,v) \mapsto s_\beta(u,v)$ and $(u,v) \mapsto t_\beta(u,v)$ are of class $C^1$.
	\end{enumerate}
	\label{Lemma4.1变分结构 4.3.1}
\end{lem}
\begin{proof}
	$(\romannumeral1)$Assume, for the sake of contradiction, that  $(u,v)\in \mathcal{P}_\beta^0(\rho_1,\rho_2)$. It follows that
	\begin{align*}
		(\Psi_{(u,v)}^\beta)'(1)=&|\nabla u|_2^2+|\nabla v|_2^2-\frac{\gamma_{r_1}}{r_1}\lambda_1 \int_{\mathbb{R}^N}(I_\alpha*|u|^{r_1})|u|^{r_1}
		-\frac{\gamma_{r_2}}{r_2}\lambda_2 \int_{\mathbb{R}^N}(I_\alpha*|v|^{r_2})|v|^{r_2}\\
		&-\beta(\gamma_p+\gamma_q)\int_{\mathbb{R}^N}(I_\alpha*|u|^p)|v|^q=0,
	\end{align*}
	\begin{align*}
		(&\Psi_{(u,v)}^\beta)''(1)=|\nabla u|_2^2+|\nabla v|_2^2
		-\beta(\gamma_p+\gamma_q)(\gamma_p+\gamma_q-1)\int_{\mathbb{R}^N}(I_\alpha*|u|^p)|v|^q\\
		&-\frac{\gamma_{r_1}}{r_1}\lambda_1 (2\gamma_{r_1}-1)\int_{\mathbb{R}^N}(I_\alpha*|u|^{r_1})|u|^{r_1}
		-\frac{\gamma_{r_2}}{r_2}\lambda_2(2\gamma_{r_2}-1)
		\int_{\mathbb{R}^N}(I_\alpha*|v|^{r_2})|v|^{r_2}=0.
	\end{align*}
	By combining the preceding two equations, we derive
	\begin{align}
		\frac{\gamma_{r_1}}{r_1}\lambda_1 (2-2\gamma_{r_1})&\int_{\mathbb{R}^N}(I_\alpha*|u|^{r_1})|u|^{r_1}
		+\beta(\gamma_p+\gamma_q)(2-\gamma_p-\gamma_q)\int_{\mathbb{R}^N}(I_\alpha*|u|^p)|v|^q \notag\\
		&=\frac{\gamma_{r_2}}{r_2}\lambda_2(2\gamma_{r_2}-2)\int_{\mathbb{R}^N}(I_\alpha*|v|^{r_2})|v|^{r_2}.
		\label{4.5 4.3.1}
	\end{align}
	On the one hand, by \eqref{GN} and \eqref{4.5 4.3.1}, we have
\begin{align*}
	|\nabla u|_2^2+|\nabla v|_2^2&=\frac{2-2\gamma_{r_1}}{2-2\gamma_{r_1}}\frac{\gamma_{r_1}}{r_1}\lambda_1 \int_{\mathbb{R}^N}(I_\alpha*|u|^{r_1})|u|^{r_1}
	+\frac{\gamma_{r_2}}{r_2}\lambda_2 \int_{\mathbb{R}^N}(I_\alpha*|v|^{r_2})|v|^{r_2}\\
	&\quad +\frac{2-\gamma_p-\gamma_q}{2-\gamma_p-\gamma_q}\beta(\gamma_p+\gamma_q)\int_{\mathbb{R}^N}(I_\alpha*|u|^p)|v|^q\\
	&\leq \frac{\gamma_{r_2}}{r_2}\lambda_2 \int_{\mathbb{R}^N}(I_\alpha*|v|^{r_2})|v|^{r_2}+
	\frac{1}{2-2\gamma_{r_1}}[\frac{\gamma_{r_1}}{r_1}\lambda_1 (2-2\gamma_{r_1})\int_{\mathbb{R}^N}(I_\alpha*|u|^{r_1})|u|^{r_1}\\
	&\quad +\beta(\gamma_p+\gamma_q)(2-\gamma_p-\gamma_q)\int_{\mathbb{R}^N}(I_\alpha*|u|^p)|v|^q ]\\
	&=\frac{\gamma_{r_2}}{r_2}\lambda_2 \int_{\mathbb{R}^N}(I_\alpha*|v|^{r_2})|v|^{r_2}+\frac{1}{2-2\gamma_{r_1}}
	\frac{\gamma_{r_2}}{r_2}\lambda_2(2\gamma_{r_2}-2)\int_{\mathbb{R}^N}(I_\alpha*|v|^{r_2})|v|^{r_2}\\
	&=\frac{\gamma_{r_2}-\gamma_{r_1}}{1-\gamma_{r_1}}\frac{\gamma_{r_2}}{r_2}\lambda_2 \int_{\mathbb{R}^N}(I_\alpha*|v|^{r_2})|v|^{r_2}\\
	&\leq \frac{\gamma_{r_2}-\gamma_{r_1}}{1-\gamma_{r_1}}2\gamma_{r_2}A_2(	|\nabla u|_2^2+|\nabla v|_2^2)^{\gamma_{r_2}}.
\end{align*}
	On the other hand, by \eqref{GN} and \eqref{4.5 4.3.1}, we obtain
\begin{align*}
	|\nabla u|_2^2+|\nabla v|_2^2&=\frac{\gamma_{r_1}}{r_1}\lambda_1 \int_{\mathbb{R}^N}(I_\alpha*|u|^{r_1})|u|^{r_1}
	+\frac{\gamma_{r_2}}{r_2}\lambda_2 \int_{\mathbb{R}^N}(I_\alpha*|v|^{r_2})|v|^{r_2}\\
	&\quad +\beta(\gamma_p+\gamma_q)\int_{\mathbb{R}^N}(I_\alpha*|u|^p)|v|^q\\
	&=\frac{\gamma_{r_2}-\gamma_{r_1}}{\gamma_{r_2}-1}\frac{\gamma_{r_1}}{r_1}\lambda_1 \int_{\mathbb{R}^N}(I_\alpha*|u|^{r_1})|u|^{r_1}\\
	&\quad +\frac{2\gamma_{r_2}-\gamma_p-\gamma_q}{2\gamma_{r_2}-2}\beta(\gamma_p+\gamma_q)\int_{\mathbb{R}^N}(I_\alpha*|u|^p)|v|^q\\
	&\leq \frac{\gamma_{r_2}-\gamma_{r_1}}{\gamma_{r_2}-1}2\gamma_{r_1}A_1(	|\nabla u|_2^2+|\nabla v|_2^2)^{\gamma_{r_1}}\\
	& +\beta\frac{2\gamma_{r_2}-\gamma_p-\gamma_q}{2\gamma_{r_2}-2}(\gamma_p+\gamma_q)
	C(N,p,q)(\rho_1^2+\rho_2^2)^{\frac{p+q-\gamma_p-\gamma_q}{2}}(	|\nabla u|_2^2+|\nabla v|_2^2)^{\frac{\gamma_p+\gamma_q}{2}}.
\end{align*}
Combining the two inequalities above together with $\beta<\beta_0$, we arrive at a contradiction. 
Hence, $\mathcal{P}_\beta^0(\rho_1,\rho_2)=\emptyset$. 
Then by an argument analogous to the proof of Lemma \ref{Lemma3.1变分结构} $(\romannumeral1)$, 
we derive $\mathcal{P}_\beta(\rho_1,\rho_2)$ 
forms a submanifold of $H$ with codimension 3; hence, we omit it here.\\
$(\romannumeral2)$-$(\romannumeral4)$ Given $(u,v) \in \mathcal{S}(\rho_1,\rho_2)$, recall that
\[\Psi_{(u,v)}^\beta(t)=\mathcal{J}_\beta(t\diamond(u,v)) \geq h \left( t(	|\nabla u|_2^2+|\nabla v|_2^2)^{\frac{1}{2}} \right),\]
and
\[\Psi_{(u,v)}^\beta(t)> \kappa\rho_1\rho_2,  \qquad \text{for any} ~~ t \in (\frac{T_0}{ (	|\nabla u|_2^2+|\nabla v|_2^2)^{\frac{1}{2}}}, 
\frac{T_1}{ (	|\nabla u|_2^2+|\nabla v|_2^2)^{\frac{1}{2}}}).   \]
We claim that $(\Psi_{(u,v)}^\beta)'(t)=0$ has at most two solutions. A direct calculation yields
\begin{align}
	\varphi(t):=& t^{2-\gamma_p-\gamma_q}(|\nabla u|_2^2+|\nabla v|_2^2)-\frac{\gamma_{r_1}}{r_1}\lambda_1 t^{2\gamma_{r_1}-\gamma_p-\gamma_q}
	\int_{\mathbb{R}^N}(I_\alpha*|u|^{r_1})|u|^{r_1}
	\label{(Psi_{(u,v)}beta)'(t)0 4.6}\\
	&-\frac{\gamma_{r_2}}{r_2}\lambda_2 t^{2\gamma_{r_2}-\gamma_p-\gamma_q}\int_{\mathbb{R}^N}(I_\alpha*|v|^{r_2})|v|^{r_2}=
	\beta (\gamma_p+\gamma_q)\int_{\mathbb{R}^N}(I_\alpha*|u|^p)|v|^q, \notag
\end{align}
and
\begin{align*}
	\varphi'(t)=& (2-\gamma_p-\gamma_q)t^{1-\gamma_p-\gamma_q}(|\nabla u|_2^2+|\nabla v|_2^2)\\
	&-(2\gamma_{r_1}-\gamma_p-\gamma_q)\frac{\gamma_{r_1}}{r_1}\lambda_1 t^{2\gamma_{r_1}-\gamma_p-\gamma_q-1}
	\int_{\mathbb{R}^N}(I_\alpha*|u|^{r_1})|u|^{r_1}\\
	& -(2\gamma_{r_2}-\gamma_p-\gamma_q)\frac{\gamma_{r_2}}{r_2}\lambda_2 t^{2\gamma_{r_2}-\gamma_p-\gamma_q-1}\int_{\mathbb{R}^N}(I_\alpha*|v|^{r_2})|v|^{r_2}\\
	=&t^{2\gamma_{r_1}-\gamma_p-\gamma_q-1}[(2-\gamma_p-\gamma_q)t^{2-2\gamma_{r_1}}(|\nabla u|_2^2+|\nabla v|_2^2)\\
	&-(2\gamma_{r_1}-\gamma_p-\gamma_q)\frac{\gamma_{r_1}}{r_1}\lambda_1 
	\int_{\mathbb{R}^N}(I_\alpha*|u|^{r_1})|u|^{r_1}\\
	& -(2\gamma_{r_2}-\gamma_p-\gamma_q)\frac{\gamma_{r_2}}{r_2}\lambda_2
	 t^{2\gamma_{r_2}-2\gamma_{r_1}}\int_{\mathbb{R}^N}(I_\alpha*|v|^{r_2})|v|^{r_2}].
\end{align*}
We proceed by defining
{\small
\begin{align*}
	\psi(t):=(2-\gamma_p-\gamma_q)t^{2-2\gamma_{r_1}}(|\nabla u|_2^2+|\nabla v|_2^2)
	-(2\gamma_{r_2}-\gamma_p-\gamma_q)\frac{\gamma_{r_2}}{r_2}\lambda_2
	t^{2\gamma_{r_2}-2\gamma_{r_1}}\int_{\mathbb{R}^N}(I_\alpha*|v|^{r_2})|v|^{r_2}.
\end{align*}
}
and
\begin{align*}
	t_0:=\left(\frac{(2-\gamma_p-\gamma_q)(|\nabla u|_2^2+|\nabla v|_2^2)}
	{(2\gamma_{r_2}-\gamma_p-\gamma_q)\frac{\gamma_{r_2}}{r_2}\lambda_2\int_{\mathbb{R}^N}(I_\alpha*|v|^{r_2})|v|^{r_2}}
	\right)^{\frac{1}{2\gamma_{r_2}-2}}.
\end{align*}
Direct computation shows that $\psi(t)$ strictly increases on $(0,t_0)$ while decreases on $(t_0,+\infty)$.\\
If $\psi(t_0) \leq (2\gamma_{r_1}-\gamma_p-\gamma_q)\frac{\gamma_{r_1}}{r_1}\lambda_1 
\int_{\mathbb{R}^N}(I_\alpha*|u|^{r_1})|u|^{r_1}$, we derive that $\varphi(t)$ strictly decreases on $(0,+\infty)$
and equation \eqref{(Psi_{(u,v)}beta)'(t)0 4.6} admits no solution.\\
If$\psi(t_0) > (2\gamma_{r_1}-\gamma_p-\gamma_q)\frac{\gamma_{r_1}}{r_1}\lambda_1 
\int_{\mathbb{R}^N}(I_\alpha*|u|^{r_1})|u|^{r_1}$, we derive that there exist $t_1,t_2>0$ such that
$\varphi(t)$ strictly decreases on $(0,t_1)$ and $(t_2,+\infty)$ while increases on $(t_1,t_2)$.
Noting that $\lim\limits_{t\to 0^+}\varphi(t)=0$, $\lim\limits_{t\to +\infty}\varphi(t)=-\infty$,
we have equation \eqref{(Psi_{(u,v)}beta)'(t)0 4.6} admits at most two solutions. Hence, the claim holds true.\\
Note that $\lim\limits_{t\to 0^+}\Psi_{(u,v)}^\beta(t)=-\kappa\int_{\mathbb{R}^N}uv \leq \kappa\rho_1\rho_2$,  $\lim\limits_{t\to +\infty}\Psi_{(u,v)}^\beta(t)=-\infty$
and $\Psi_{(u,v)}^\beta(t)$ is strictly decreasing for $0<t\ll 1$. Thus, $\Psi_{(u,v)}^\beta(t)$ has exactly two critical points $s_\beta(u,v)<t_\beta(u,v)$ satisfying
\[\Psi_{(u,v)}^\beta(s_\beta(u,v))=\min\{\Psi_{(u,v)}^\beta(t):t(|\nabla u|_2^2+|\nabla v|_2^2)^{\frac{1}{2}}\leq T_0\} \]
and  $\Psi_{(u,v)}^\beta(t_\beta(u,v))=\max\limits_{t>0}\Psi_{(u,v)}^\beta(t)$.\\
$(\romannumeral5)$ By an argument analogous to the proof of Lemma \ref{Lemma3.1变分结构} $(\romannumeral5)$, 
we derive that the maps $(u,v) \mapsto s_\beta(u,v)$ and $(u,v) \mapsto t_\beta(u,v)$ are of class $C^1$; 
hence, we omit it here.
\end{proof}
As the proof parallels the preceding case, we conclude the same existence result holds (Theorem \ref{Theorem1.9}).

\subsubsection{$\frac{2N+2\alpha}{N}<p+q=2r_1<\frac{2N+2\alpha+4}{N}<2r_2<\frac{2N+2\alpha}{N-2}$}\leavevmode\par
Under the current assumption, we have $0<\gamma_p+\gamma_q=2\gamma_{r_1}<2<2\gamma_{r_2}<\frac{2N+2\alpha}{N-2}$ and
\begin{align*}
	h(s)=&\frac{1}{2}s^2-A_1s^{2\gamma_{r_1}}-A_2 s^{2\gamma_{r_2}}-A_3 s^{\gamma_p+\gamma_q}-\kappa \rho_1 \rho_2\\
	=&\frac{1}{2}s^2-(A_1+A_3)s^{2\gamma_{r_1}}-A_2 s^{2\gamma_{r_2}}-\kappa \rho_1 \rho_2,
\end{align*}
where $h(s)$ is given by \eqref{h(s)}. 
A direct calculation yields
\begin{align*}
	h'(s)&=s-2\gamma_{r_1}(A_1+A_3)s^{2\gamma_{r_1}-1}-2\gamma_{r_2}A_2s^{2\gamma_{r_2}-1}\\
	&=s^{2\gamma_{r_1}-1}(s^{2-2\gamma_{r_1}}-2\gamma_{r_1}(A_1+A_3)-2\gamma_{r_2}A_2s^{2\gamma_{r_2}-2\gamma_{r_1}}).
\end{align*}
Set
\begin{equation*}
	g(s):=s^{2-2\gamma_{r_1}}-2\gamma_{r_2}A_2s^{2\gamma_{r_2}-2\gamma_{r_1}},
\end{equation*}
and we derive
\begin{align*}
	g'(s)=&(2-2\gamma_{r_1})s^{1-2\gamma_{r_1}}-(2\gamma_{r_2}-2\gamma_{r_1})2\gamma_{r_2}A_2s^{2\gamma_{r_2}-2\gamma_{r_1}-1}\\
	=&2s^{1-2\gamma_{r_1}}(1-\gamma_{r_1}-(\gamma_{r_2}-\gamma_{r_1})2\gamma_{r_2}A_2s^{2\gamma_{r_2}-2}).
\end{align*}
Defining 
\[s_0= \left[ \frac{1-\gamma_{r_1}}{(\gamma_{r_2}-\gamma_{r_1})2\gamma_{r_2}A_2} \right] ^{\frac{1}{2\gamma_{r_2}-2}},\]
we derive
$g(s)$ strictly increases on $(0,s_0)$ while decreases on $(s_0,+\infty)$. 
A direct calculation yields
\begin{align*}
	g(s_0)=(2\gamma_{r_2}A_2)^{-\frac{1-\gamma_{r_1}}{\gamma_{r_2}-1}}
	\left[(\frac{1-\gamma_{r_1}}{\gamma_{r_2}-\gamma_{r_1}})^{\frac{1-\gamma_{r_1}}{\gamma_{r_2}-1}}-
	(\frac{1-\gamma_{r_1}}{\gamma_{r_2}-\gamma_{r_1}})^{\frac{\gamma_{r_2}-\gamma_{r_1}}{\gamma_{r_2}-1}}\right].
\end{align*}
Assume that $\lambda_1, \lambda_2, \rho_1,\rho_2 > 0$ satisfying 
\begin{equation}
	\left\{
	\begin{aligned}
		s_0^2-2\gamma_{r_1}A_1s_0^{2\gamma_{r_1}}-2\gamma_{r_2}A_2 s_0^{2\gamma_{r_2}}>&0,\\
		\frac{1}{2}s_0^2-A_1s_0^{2\gamma_{r_1}}-A_2 s_0^{2\gamma_{r_2}}>&0,\\
		(\frac{\gamma_{r_2}-\gamma_{r_1}}{\gamma_{r_2}-1}2\gamma_{r_1}A_1)^{\gamma_{r_2}-1}
		(\frac{\gamma_{r_2}-\gamma_{r_1}}{1-\gamma_{r_1}}2\gamma_{r_2}A_2)^{1-\gamma_{r_1}}<&1.
	\end{aligned}
	\right.
	\label{g(s0)>0 assumption4.8}
\end{equation}
Direct computation shows that \eqref{g(s0)>0 assumption4.8} holds for sufficiently small $A_1, A_2>0$, 
justifying \eqref{g(s0)>0 assumption4.8}  as a reasonable assumption.
We proceed by defining
\begin{align*}
	\beta_2:=&\frac{s_0^{2-2\gamma_{r_1}}-2\gamma_{r_1}A_1-
		2\gamma_{r_2}A_2s_0^{2\gamma_{r_2}-2\gamma_{r_1}}}
	{C(N,p,q)2\gamma_{r_1}(\rho_1^2+\rho_2^2)^{\frac{p+q-\gamma_p-\gamma_q}{2}}},\\
	\beta_3:=&\frac{\frac{1}{2}s_0^{2-2\gamma_{r_1}}-A_1
		-A_2s_0^{2\gamma_{r_2}-2\gamma_{r_1}}}
	{C(N,p,q)(\rho_1^2+\rho_2^2)^{\frac{p+q-\gamma_p-\gamma_q}{2}}},
\end{align*}
and
\begin{equation}
	\beta_0:=\min\{\beta_1,\beta_2, \beta_3\},
	\label{beta 0 in 4.3.2}
\end{equation}
where $\beta_1$ is defined by \eqref{beta_1 4.3.1}. 
We also denote
\begin{equation}
	\kappa_0:=\frac{\frac{1}{2}s_0^2-(A_1+\overline{A}_3)s_0^{2\gamma_{r_1}}
		-A_2s_0^{2\gamma_{r_2}}}{2\rho_1\rho_2},
	\label{kappa 0 in 4.3.2}
\end{equation}
where $\overline{A}_3:=\beta_0 C(N,p,q)(\rho_1^2+\rho_2^2)^{\frac{p+q-\gamma_p-\gamma_q}{2}}.$

Assuming $0<\beta<\beta_0$, $0<\kappa<\kappa_0$, then direct computation shows that $g(s_0)>2\gamma_{r_1}(A_1+A_3)$,
$h(s_0)>\kappa \rho_1\rho_2>0$. Hence, $h(s)$ possesses precisely two critical points $0<s_1<s_0<s_2$ satisfying
\[h(s_1)=\min\limits_{0<s<s_0} h(s)<0,\quad h(s_2)=\max_{s>0}h(s)>\kappa\rho_1\rho_2.
\]
In addition, we can find $T_1>T_0>0$ such that $h(T_0)=h(T_1)=\kappa \rho_1 \rho_2$ 
and $h(s)>\kappa \rho_1 \rho_2$ exactly when $s \in (T_0,T_1)$.

As the proof parallels the preceding case, we conclude the same existence result holds (Theorem \ref{Theorem1.10}).

\subsubsection{$\frac{2N+2\alpha}{N}<p+q<2r_1=\frac{2N+2\alpha+4}{N}<2r_2<\frac{2N+2\alpha}{N-2}$}\leavevmode\par
Under the current assumption, we have $0<\gamma_p+\gamma_q<2\gamma_{r_1}=2<2\gamma_{r_2}<\frac{2N+2\alpha}{N-2}$ and
\begin{align*}
	h(s)=&\frac{1}{2}s^2-A_1s^{2\gamma_{r_1}}-A_2 s^{2\gamma_{r_2}}-A_3 s^{\gamma_p+\gamma_q}-\kappa \rho_1 \rho_2\\
	=&(\frac{1}{2}-A_1)s^2-A_2 s^{2\gamma_{r_2}}-A_3 s^{\gamma_p+\gamma_q}-\kappa \rho_1 \rho_2,
\end{align*}
where $h(s)$ is given by \eqref{h(s)}. 
Assume that $\lambda_1, \rho_1> 0$ satisfying 
\begin{equation}
	\frac{1}{2}-A_1>0.
	\label{frac{1}{2}-A_1>0 4.3.3}
\end{equation}
A direct calculation yields
\begin{align*}
	h'(s)&=(1-2A_1)s-2\gamma_{r_2}A_2 s^{2\gamma_{r_2}-1}-(\gamma_p+\gamma_q)A_3 s^{\gamma_p+\gamma_q-1}\\
	&= s^{\gamma_p+\gamma_q-1}\left[(1-2A_1)s^{2-\gamma_p-\gamma_q}-2\gamma_{r_2}A_2
	 s^{2\gamma_{r_2}-\gamma_p-\gamma_q}-(\gamma_p+\gamma_q)A_3\right].
\end{align*}
Set
\begin{equation*}
	g(s):=(1-2A_1)s^{2-\gamma_p-\gamma_q}-2\gamma_{r_2}A_2s^{2\gamma_{r_2}-\gamma_p-\gamma_q},
\end{equation*}
and we derive
\begin{align*}
	g'(s)=&(2-\gamma_p-\gamma_q)(1-2A_1)s^{1-\gamma_p-\gamma_q}-
	(2\gamma_{r_2}-\gamma_p-\gamma_q)2\gamma_{r_2}A_2s^{2\gamma_{r_2}-\gamma_p-\gamma_q-1}\\
	=&s^{1-\gamma_p-\gamma_q}\left[ (2-\gamma_p-\gamma_q)(1-2A_1)-
	(2\gamma_{r_2}-\gamma_p-\gamma_q)2\gamma_{r_2}A_2s^{2\gamma_{r_2}-2} \right].
\end{align*}
Defining 
\[s_0= \left[ \frac{(2-\gamma_p-\gamma_q)(1-2A_1)}{(2\gamma_{r_2}-\gamma_p-\gamma_q)2\gamma_{r_2}A_2} \right] ^{\frac{1}{2\gamma_{r_2}-2}},\]
we derive
$g(s)$ strictly increases on $(0,s_0)$ while decreases on $(s_0,+\infty)$. 
Note that
\begin{align*}
	(\frac{1}{2}-A_1)s_0^2-A_2 s_0^{2\gamma_{r_2}}=&\frac{1}{2}s_0^{\gamma_p+\gamma_q}
	((1-2A_1)s_0^{2-\gamma_p-\gamma_q}-2A_2s_0^{2\gamma_{r_2}-\gamma_p-\gamma_q})\\
	\geq&\frac{1}{2}s_0^{\gamma_p+\gamma_q}
	((1-2A_1)s_0^{2-\gamma_p-\gamma_q}-2\gamma_{r_2}A_2s_0^{2\gamma_{r_2}-\gamma_p-\gamma_q})\\
	=&\frac{1}{2}s_0^{\gamma_p+\gamma_q}g(s_0)>0.
\end{align*}
We proceed by defining
\begin{align*}
	\beta_1:=&\frac{(1-2A_1)s_0^{2-\gamma_p-\gamma_q}-2\gamma_{r_2}A_2s_0^{2\gamma_{r_2}-\gamma_p-\gamma_q}}
	{C(N,p,q)(\gamma_p+\gamma_q)(\rho_1^2+\rho_2^2)^{\frac{p+q-\gamma_p-\gamma_q}{2}}},\\
	\beta_2:=&\frac{(\frac{1}{2}-A_1)s^{2-\gamma_p-\gamma_q}-A_2 s^{2\gamma_{r_2}-\gamma_p-\gamma_q}}
	{C(N,p,q)(\rho_1^2+\rho_2^2)^{\frac{p+q-\gamma_p-\gamma_q}{2}}},
\end{align*}
and
\begin{equation}
	\beta_0:=\min\{\beta_1,\beta_2\}.
	\label{beta 0 in 4.3.3}
\end{equation}
We also denote
\begin{equation}
	\kappa_0:=\frac{(\frac{1}{2}-A_1)s_0^2-A_2 s_0^{2\gamma_{r_2}}-\overline{A}_3 s_0^{\gamma_p+\gamma_q}}{2\rho_1\rho_2},
	\label{kappa 0 in 4.3.3}
\end{equation}
where $\overline{A}_3:=\beta_0 C(N,p,q)(\rho_1^2+\rho_2^2)^{\frac{p+q-\gamma_p-\gamma_q}{2}}.$

Assuming $0<\beta<\beta_0$, $0<\kappa<\kappa_0$, then direct computation shows that $g(s_0)>(\gamma_p+\gamma_q)A_3$,
$h(s_0)>\kappa \rho_1\rho_2>0$. Hence, $h(s)$ possesses precisely two critical points $0<s_1<s_0<s_2$ satisfying
\[h(s_1)=\min\limits_{0<s<s_0} h(s)<0,\quad h(s_2)=\max_{s>0}h(s)>\kappa\rho_1\rho_2.
\]
In addition, we can find $T_1>T_0>0$ such that $h(T_0)=h(T_1)=\kappa \rho_1 \rho_2$ 
and $h(s)>\kappa \rho_1 \rho_2$ exactly when $s \in (T_0,T_1)$.
\begin{lem}
	Assume that $(u,v) \in \mathcal{S}(\rho_1,\rho_2)$, $\lambda_1, \rho_1> 0$ satisfying \eqref{frac{1}{2}-A_1>0 4.3.3}, $0<\beta<\beta_0$, $0<\kappa<\kappa_0$, where $\beta_0$ and $\kappa_0$ are defined in \eqref{beta 0 in 4.3.3} and \eqref{kappa 0 in 4.3.3} respectively, then $\Psi_{(u,v)}^\beta(t)$ has exactly two critical points $s_\beta(u,v)<t_\beta(u,v)$. Furthermore,
	\begin{enumerate}[label=(\roman*)]
		\item $\mathcal{P}_\beta^0(\rho_1,\rho_2)=\emptyset$ and $\mathcal{P}_\beta(\rho_1,\rho_2)$ 
		forms a submanifold of $H$ with codimension 3.
		\item $s_\beta(u,v)\diamond (u,v) \in \mathcal{P}_\beta^+(\rho_1,\rho_2)$, $t_\beta(u,v)\diamond (u,v) \in \mathcal{P}_\beta^-(\rho_1,\rho_2)$ and $t \diamond (u,v) \in \mathcal{P}_\beta(\rho_1,\rho_2)$ 
		if and only if $t=s_\beta(u,v) $ or $ t_\beta(u,v)$.
		\item $\Psi_{(u,v)}^\beta(s_\beta(u,v))=\min\{\Psi_{(u,v)}^\beta(t):t(|\nabla u|_2^2+|\nabla v|_2^2)^{\frac{1}{2}}\leq T_0\}$.
		\item $\Psi_{(u,v)}^\beta(t)$ decreases strictly on $(t_\beta(u,v),+\infty)$ and  \[\Psi_{(u,v)}^\beta(t_\beta(u,v))=\max\limits_{t>0}\Psi_{(u,v)}^\beta(t).\]
		\item The maps $(u,v) \mapsto s_\beta(u,v)$ and $(u,v) \mapsto t_\beta(u,v)$ are of class $C^1$.
	\end{enumerate}
	\label{Lemma4.2变分结构 4.3.3}
\end{lem}
\begin{proof}
	Assume, for the sake of contradiction, that  $(u,v)\in \mathcal{P}_\beta^0(\rho_1,\rho_2)$. It follows that
	\begin{align*}
		(\Psi_{(u,v)}^\beta)'(1)=&|\nabla u|_2^2+|\nabla v|_2^2-\frac{\gamma_{r_1}}{r_1}\lambda_1 \int_{\mathbb{R}^N}(I_\alpha*|u|^{r_1})|u|^{r_1}
		-\frac{\gamma_{r_2}}{r_2}\lambda_2 \int_{\mathbb{R}^N}(I_\alpha*|v|^{r_2})|v|^{r_2}\\
		&-\beta(\gamma_p+\gamma_q)\int_{\mathbb{R}^N}(I_\alpha*|u|^p)|v|^q=0,
	\end{align*}
	\begin{align*}
		(&\Psi_{(u,v)}^\beta)''(1)=|\nabla u|_2^2+|\nabla v|_2^2
		-\beta(\gamma_p+\gamma_q)(\gamma_p+\gamma_q-1)\int_{\mathbb{R}^N}(I_\alpha*|u|^p)|v|^q\\
		&-\frac{\gamma_{r_1}}{r_1}\lambda_1 (2\gamma_{r_1}-1)\int_{\mathbb{R}^N}(I_\alpha*|u|^{r_1})|u|^{r_1}
		-\frac{\gamma_{r_2}}{r_2}\lambda_2(2\gamma_{r_2}-1)
		\int_{\mathbb{R}^N}(I_\alpha*|v|^{r_2})|v|^{r_2}=0.
	\end{align*}
	By combining the preceding two equations, we derive
	\begin{align}
		\beta(\gamma_p+\gamma_q)(2-\gamma_p-\gamma_q)\int_{\mathbb{R}^N}(I_\alpha*|u|^p)|v|^q 
		=\frac{\gamma_{r_2}}{r_2}\lambda_2(2\gamma_{r_2}-2)\int_{\mathbb{R}^N}(I_\alpha*|v|^{r_2})|v|^{r_2}.
		\label{4.14 4.3.3}
	\end{align}
	On the one hand, by \eqref{GN} and \eqref{4.14 4.3.3}, we have
	\begin{align*}
		|\nabla u|_2^2+|\nabla v|_2^2&=\frac{\gamma_{r_1}}{r_1}\lambda_1 \int_{\mathbb{R}^N}(I_\alpha*|u|^{r_1})|u|^{r_1}
		+\frac{\gamma_{r_2}}{r_2}\lambda_2 \int_{\mathbb{R}^N}(I_\alpha*|v|^{r_2})|v|^{r_2}\\
		&\quad +\beta(\gamma_p+\gamma_q)\int_{\mathbb{R}^N}(I_\alpha*|u|^p)|v|^q\\
		&=\frac{\gamma_{r_1}}{r_1}\lambda_1 \int_{\mathbb{R}^N}(I_\alpha*|u|^{r_1})|u|^{r_1}+
		\frac{2\gamma_{r_2}-\gamma_p-\gamma_q}{2-\gamma_p-\gamma_q}
		\frac{\gamma_{r_2}}{r_2}\lambda_2 \int_{\mathbb{R}^N}(I_\alpha*|v|^{r_2})|v|^{r_2}\\
		&\leq 2\gamma_{r_1}A_1(|\nabla u|_2^2+|\nabla v|_2^2)^{\gamma_{r_1}}
		+\frac{2\gamma_{r_2}-\gamma_p-\gamma_q}{2-\gamma_p-\gamma_q}
		2\gamma_{r_2}A_2(|\nabla u|_2^2+|\nabla v|_2^2)^{\gamma_{r_2}}.
	\end{align*}
	On the other hand, by \eqref{GN} and \eqref{4.14 4.3.3}, we obtain
	\begin{align*}
		|\nabla u|_2^2+|\nabla v|_2^2&=\frac{\gamma_{r_1}}{r_1}\lambda_1 \int_{\mathbb{R}^N}(I_\alpha*|u|^{r_1})|u|^{r_1}
		+\frac{\gamma_{r_2}}{r_2}\lambda_2 \int_{\mathbb{R}^N}(I_\alpha*|v|^{r_2})|v|^{r_2}\\
		&\quad +\beta(\gamma_p+\gamma_q)\int_{\mathbb{R}^N}(I_\alpha*|u|^p)|v|^q\\
		&=\frac{\gamma_{r_1}}{r_1}\lambda_1 \int_{\mathbb{R}^N}(I_\alpha*|u|^{r_1})|u|^{r_1}
		+\frac{2\gamma_{r_2}-\gamma_p-\gamma_q}{2\gamma_{r_2}-2}
		\beta(\gamma_p+\gamma_q)\int_{\mathbb{R}^N}(I_\alpha*|u|^p)|v|^q\\
		&\leq 2\gamma_{r_1}A_1(|\nabla u|_2^2+|\nabla v|_2^2)^{\gamma_{r_1}}+
		\frac{2\gamma_{r_2}-\gamma_p-\gamma_q}{2\gamma_{r_2}-2}
		(\gamma_p+\gamma_q)A_3(|\nabla u|_2^2+|\nabla v|_2^2)^{\frac{\gamma_p+\gamma_q}{2}}.
	\end{align*}
	By combining the two inequalities above, we get
	\[ \beta \geq \frac{(2\gamma_{r_2}-2)(1-2A_1)s_0^{2-\gamma_p-\gamma_q}}{(2\gamma_{r_2}-\gamma_p-\gamma_q)
		C(N,p,q)(\gamma_p+\gamma_q)(\rho_1^2+\rho_2^2)^{\frac{p+q-\gamma_p-\gamma_q}{2}}}=\beta_1,    \]
	which contradicts $\beta<\beta_0$.
	
		By an argument analogous to the proof of Lemma \ref{Lemma4.1变分结构 4.3.1}                 , 
	we derive that the remaining conclusions hold; 
	hence, we omit it here.
\end{proof}
As the proof parallels the preceding case, we conclude the same existence result holds (Theorem \ref{Theorem1.11}).

\subsubsection{$\frac{2N+2\alpha}{N}<p+q=2r_1=\frac{2N+2\alpha+4}{N}<2r_2<\frac{2N+2\alpha}{N-2}$}\leavevmode\par
Under the current assumption, we have $0<\gamma_p+\gamma_q=2\gamma_{r_1}=2<2\gamma_{r_2}<\frac{2N+2\alpha}{N-2}$ and
\begin{align*}
	h(s)=&\frac{1}{2}s^2-A_1s^{2\gamma_{r_1}}-A_2 s^{2\gamma_{r_2}}-A_3 s^{\gamma_p+\gamma_q}-\kappa \rho_1 \rho_2\\
	=&(\frac{1}{2}-(A_1+A_3))s^2-A_2 s^{2\gamma_{r_2}}-\kappa \rho_1 \rho_2,
\end{align*}
where $h(s)$ is given by \eqref{h(s)}. 
Assume that $\lambda_1, \rho_1,\beta> 0$ satisfying 
\begin{equation}
	\frac{1}{2}-(A_1+A_3)>0.
	\label{frac{1}{2}-(A_1+A_3)>0 4.3.4}
\end{equation}
A direct calculation yields
\begin{align*}
	h'(s)&=(1-2(A_1+A_3))s-2\gamma_{r_2}A_2 s^{2\gamma_{r_2}-1}.
\end{align*}
Defining 
\[s_0= \left[ \frac{1-2(A_1+A_3)}{2\gamma_{r_2}A_2 } \right] ^{\frac{1}{2\gamma_{r_2}-2}},\]
we derive that
$h(s)$ strictly increases on $(0,s_0)$ while decreases on $(s_0,+\infty)$.
As the proof parallels the preceding case ($p+q =\frac{2N+2\alpha+4}{N}< 2r_1=2r_2<\frac{2N+2\alpha}{N-2}$), we conclude the same existence result holds (Theorem \ref{Theorem1.12}).

\subsubsection{$\frac{2N+2\alpha}{N}<p+q<\frac{2N+2\alpha+4}{N}<2r_1<2r_2<\frac{2N+2\alpha}{N-2}$}\leavevmode\par
Under the current assumption, we have $0<\gamma_p+\gamma_q<2<2\gamma_{r_1}<2\gamma_{r_2}<\frac{2N+2\alpha}{N-2}$ and
\begin{align*}
	h(s)=\frac{1}{2}s^2-A_1s^{2\gamma_{r_1}}-A_2 s^{2\gamma_{r_2}}-A_3 s^{\gamma_p+\gamma_q}-\kappa \rho_1 \rho_2,
\end{align*}
where $h(s)$ is given by \eqref{h(s)}. 
A direct calculation yields
\begin{align*}
	h'(s)&=s-2\gamma_{r_1}A_1s^{2\gamma_{r_1}-1}-2\gamma_{r_2}A_2 s^{2\gamma_{r_2}-1}
	-(\gamma_p+\gamma_q)A_3 s^{\gamma_p+\gamma_q-1}\\
	&=s^{\gamma_p+\gamma_q-1} \left[s^{2-\gamma_p-\gamma_q}-2\gamma_{r_1}A_1s^{2\gamma_{r_1}-\gamma_p-\gamma_q}
	-2\gamma_{r_2}A_2 s^{2\gamma_{r_2}-\gamma_p-\gamma_q}
	-(\gamma_p+\gamma_q)A_3 \right].
\end{align*}
Set
\begin{equation*}
	g(s):=s^{2-\gamma_p-\gamma_q}-2\gamma_{r_1}A_1s^{2\gamma_{r_1}-\gamma_p-\gamma_q}
	-2\gamma_{r_2}A_2 s^{2\gamma_{r_2}-\gamma_p-\gamma_q},
\end{equation*}
and we derive
\begin{align*}
	g'(s)=&(2-\gamma_p-\gamma_q)s^{1-\gamma_p-\gamma_q}
	-(2\gamma_{r_1}-\gamma_p-\gamma_q)2\gamma_{r_1}A_1s^{2\gamma_{r_1}-\gamma_p-\gamma_q-1}\\
	&-(2\gamma_{r_2}-\gamma_p-\gamma_q)2\gamma_{r_2}A_2 s^{2\gamma_{r_2}-\gamma_p-\gamma_q-1}\\
	=&s^{1-\gamma_p-\gamma_q}((2-\gamma_p-\gamma_q)-
	(2\gamma_{r_1}-\gamma_p-\gamma_q)2\gamma_{r_1}A_1s^{2\gamma_{r_1}-2}\\
	&-(2\gamma_{r_2}-\gamma_p-\gamma_q)2\gamma_{r_2}A_2 s^{2\gamma_{r_2}-2}).
\end{align*}
Now we define
\begin{align*}
	f(s):=(2\gamma_{r_1}-\gamma_p-\gamma_q)2\gamma_{r_1}A_1s^{2\gamma_{r_1}-2}
	+(2\gamma_{r_2}-\gamma_p-\gamma_q)2\gamma_{r_2}A_2 s^{2\gamma_{r_2}-2}.
\end{align*}
Direct computation shows that $\lim\limits_{s\to 0^+}f(s)=0$, $\lim\limits_{s\to +\infty}f(s)=+\infty$ and $f(s)$ is strictly increasing on $(0,+\infty)$.
It follows that there exists $s_*>0$ such that $f(s_*)=2-\gamma_p-\gamma_q$ and 
$g(s)$ strictly increases on $(0,s_*)$ while decreases on $(s_*,+\infty)$. 
Set
\[s_0:=(\frac{2-\gamma_p-\gamma_q}{(2\gamma_{r_2}-\gamma_p-\gamma_q)2\gamma_{r_2}A_2})
^{{\frac{1}{2\gamma_{r_2}-2}}}.
\]
Assume that $\lambda_1, \lambda_2, \rho_1,\rho_2 > 0$ satisfying
\begin{equation}
	1-2\gamma_{r_1}A_1s_0^{2\gamma_{r_1}-2}-2\gamma_{r_2}A_2 s_0^{2\gamma_{r_2}-2}>0.
	\label{g(s_0)>0 4.3.5}
\end{equation}
Note that
\begin{align*}
    \frac{1}{2}s_0^2-A_1s_0^{2\gamma_{r_1}}-A_2 s_0^{2\gamma_{r_2}}=&\frac{1}{2}s_0^2
    (1-2A_1s_0^{2\gamma_{r_1}-2}-2A_2 s_0^{2\gamma_{r_2}-2})\\
    \geq &\frac{1}{2}s_0^2(1-2\gamma_{r_1}A_1s_0^{2\gamma_{r_1}-2}-2\gamma_{r_2}A_2 s_0^{2\gamma_{r_2}-2})>0.
\end{align*}
We proceed by defining
\begin{align*}
	\beta_1:=&\frac{s_0^{2-\gamma_p-\gamma_q}-2\gamma_{r_1}A_1s_0^{2\gamma_{r_1}-\gamma_p-\gamma_q}
		-2\gamma_{r_2}A_2 s_0^{2\gamma_{r_2}-\gamma_p-\gamma_q}}
	{C(N,p,q)(\gamma_p+\gamma_q)(\rho_1^2+\rho_2^2)^{\frac{p+q-\gamma_p-\gamma_q}{2}}},\\
	\beta_2:=&\frac{\frac{1}{2}s_0^{2-\gamma_p-\gamma_q}-A_1s_0^{2\gamma_{r_1}-\gamma_p-\gamma_q}
		-A_2 s_0^{2\gamma_{r_2}-\gamma_p-\gamma_q}}
	{C(N,p,q)(\rho_1^2+\rho_2^2)^{\frac{p+q-\gamma_p-\gamma_q}{2}}},\\
	\beta_3:=&\frac{(2\gamma_{r_1}-2)\left[\frac{2-\gamma_p-\gamma_q}{2(2\gamma_{r_1}-\gamma_p-\gamma_q)2\gamma_{r_1}A_1}\right]^{\frac{2-\gamma_p-\gamma_q}{2\gamma_{r_1}-2}}}{(2\gamma_{r_1}-\gamma_p-\gamma_q)C(N,p,q)(\gamma_p+\gamma_q)(\rho_1^2+\rho_2^2)^{\frac{p+q-\gamma_p-\gamma_q}{2}}}\\
	\beta_4:=&\frac{(2\gamma_{r_1}-2)\left[\frac{2-\gamma_p-\gamma_q}{2(2\gamma_{r_2}-\gamma_p-\gamma_q)2\gamma_{r_2}A_2}\right]^{\frac{2-\gamma_p-\gamma_q}{2\gamma_{r_2}-2}}}{(2\gamma_{r_1}-\gamma_p-\gamma_q)C(N,p,q)(\gamma_p+\gamma_q)(\rho_1^2+\rho_2^2)^{\frac{p+q-\gamma_p-\gamma_q}{2}}},
\end{align*}
and
\begin{equation}
	\beta_0:=\min\{\beta_1,\beta_2\}.
	\label{beta 0 in 4.3.5}
\end{equation}
We also denote
\begin{equation}
	\kappa_0:=\frac{\frac{1}{2}s_0^2-A_1s_0^{2\gamma_{r_1}}-A_2 s_0^{2\gamma_{r_2}}-\overline{A}_3 s_0^{\gamma_p+\gamma_q}}{2\rho_1\rho_2},
	\label{kappa 0 in 4.3.5}
\end{equation}
where $\overline{A}_3:=\beta_0 C(N,p,q)(\rho_1^2+\rho_2^2)^{\frac{p+q-\gamma_p-\gamma_q}{2}}.$

Assuming $0<\beta<\beta_0$, $0<\kappa<\kappa_0$, then direct computation shows that $g(s_0)>(\gamma_p+\gamma_q)A_3$,
$h(s_0)>\kappa \rho_1\rho_2>0$. Hence, $h(s)$ possesses precisely two critical points $0<s_1<s_0<s_2$ satisfying
\[h(s_1)=\min\limits_{0<s<s_0} h(s)<0,\quad h(s_2)=\max_{s>0}h(s)>\kappa\rho_1\rho_2.
\]
In addition, we can find $T_1>T_0>0$ such that $h(T_0)=h(T_1)=\kappa \rho_1 \rho_2$ 
and $h(s)>\kappa \rho_1 \rho_2$ exactly when $s \in (T_0,T_1)$.
\begin{lem}
	Assume that $(u,v) \in \mathcal{S}(\rho_1,\rho_2)$, $\lambda_1, \lambda_2, \rho_1,\rho_2 > 0$ satisfying \eqref{g(s_0)>0 4.3.5}, $0<\beta<\beta_0$, $0<\kappa<\kappa_0$, where $\beta_0$ and $\kappa_0$ are defined in \eqref{beta 0 in 4.3.5} and \eqref{kappa 0 in 4.3.5} respectively, then $\Psi_{(u,v)}^\beta(t)$ has exactly two critical points $s_\beta(u,v)<t_\beta(u,v)$. Furthermore,
	\begin{enumerate}[label=(\roman*)]
		\item $\mathcal{P}_\beta^0(\rho_1,\rho_2)=\emptyset$ and $\mathcal{P}_\beta(\rho_1,\rho_2)$ 
		forms a submanifold of $H$ with codimension 3.
		\item $s_\beta(u,v)\diamond (u,v) \in \mathcal{P}_\beta^+(\rho_1,\rho_2)$, $t_\beta(u,v)\diamond (u,v) \in \mathcal{P}_\beta^-(\rho_1,\rho_2)$ and $t \diamond (u,v) \in \mathcal{P}_\beta(\rho_1,\rho_2)$ 
		if and only if $t=s_\beta(u,v) $ or $ t_\beta(u,v)$.
		\item $\Psi_{(u,v)}^\beta(s_\beta(u,v))=\min\{\Psi_{(u,v)}^\beta(t):t(|\nabla u|_2^2+|\nabla v|_2^2)^{\frac{1}{2}}\leq T_0\}$.
		\item $\Psi_{(u,v)}^\beta(t)$ decreases strictly on $(t_\beta(u,v),+\infty)$ and  \[\Psi_{(u,v)}^\beta(t_\beta(u,v))=\max\limits_{t>0}\Psi_{(u,v)}^\beta(t).\]
		\item The maps $(u,v) \mapsto s_\beta(u,v)$ and $(u,v) \mapsto t_\beta(u,v)$ are of class $C^1$.
	\end{enumerate}
	\label{Lemma4.3变分结构 4.3.5}
\end{lem}
\begin{proof}
		Assume, for the sake of contradiction, that  $(u,v)\in \mathcal{P}_\beta^0(\rho_1,\rho_2)$. It follows that
	\begin{align*}
		(\Psi_{(u,v)}^\beta)'(1)=&|\nabla u|_2^2+|\nabla v|_2^2-\frac{\gamma_{r_1}}{r_1}\lambda_1 \int_{\mathbb{R}^N}(I_\alpha*|u|^{r_1})|u|^{r_1}
		-\frac{\gamma_{r_2}}{r_2}\lambda_2 \int_{\mathbb{R}^N}(I_\alpha*|v|^{r_2})|v|^{r_2}\\
		&-\beta(\gamma_p+\gamma_q)\int_{\mathbb{R}^N}(I_\alpha*|u|^p)|v|^q=0,
	\end{align*}
	\begin{align*}
		(&\Psi_{(u,v)}^\beta)''(1)=|\nabla u|_2^2+|\nabla v|_2^2
		-\beta(\gamma_p+\gamma_q)(\gamma_p+\gamma_q-1)\int_{\mathbb{R}^N}(I_\alpha*|u|^p)|v|^q\\
		&-\frac{\gamma_{r_1}}{r_1}\lambda_1 (2\gamma_{r_1}-1)\int_{\mathbb{R}^N}(I_\alpha*|u|^{r_1})|u|^{r_1}
		-\frac{\gamma_{r_2}}{r_2}\lambda_2(2\gamma_{r_2}-1)
		\int_{\mathbb{R}^N}(I_\alpha*|v|^{r_2})|v|^{r_2}=0.
	\end{align*}
	By combining the preceding two equations, we derive
	\begin{align}
			\frac{\gamma_{r_1}}{r_1}\lambda_1 (2\gamma_{r_1}-2)&\int_{\mathbb{R}^N}(I_\alpha*|u|^{r_1})|u|^{r_1}
		+\frac{\gamma_{r_2}}{r_2}\lambda_2(2\gamma_{r_2}-2)\int_{\mathbb{R}^N}(I_\alpha*|v|^{r_2})|v|^{r_2}\notag\\
		&=\beta(\gamma_p+\gamma_q)(2-\gamma_p-\gamma_q)\int_{\mathbb{R}^N}(I_\alpha*|u|^p)|v|^q.
		\label{4.19 4.3.5}
	\end{align}
		On the one hand, by \eqref{GN} and \eqref{4.19 4.3.5}, we have
		\begin{align*}
			|\nabla u|_2^2+|\nabla v|_2^2&=\frac{\gamma_{r_1}}{r_1}\lambda_1 \int_{\mathbb{R}^N}(I_\alpha*|u|^{r_1})|u|^{r_1}
			+\frac{\gamma_{r_2}}{r_2}\lambda_2 \int_{\mathbb{R}^N}(I_\alpha*|v|^{r_2})|v|^{r_2}\\
			&\quad +\beta(\gamma_p+\gamma_q)\int_{\mathbb{R}^N}(I_\alpha*|u|^p)|v|^q\\
			&=\frac{2\gamma_{r_1}-2}{2\gamma_{r_1}-2}\frac{\gamma_{r_1}}{r_1}\lambda_1 \int_{\mathbb{R}^N}(I_\alpha*|u|^{r_1})|u|^{r_1}
			+\frac{2\gamma_{r_2}-2}{2\gamma_{r_2}-2}\frac{\gamma_{r_2}}{r_2}\lambda_2 \int_{\mathbb{R}^N}(I_\alpha*|v|^{r_2})|v|^{r_2}\\
			&\quad +\beta(\gamma_p+\gamma_q)\int_{\mathbb{R}^N}(I_\alpha*|u|^p)|v|^q\\
			&\leq \frac{1}{2\gamma_{r_1}-2}((2\gamma_{r_1}-2)\frac{\gamma_{r_1}}{r_1}\lambda_1 \int_{\mathbb{R}^N}(I_\alpha*|u|^{r_1})|u|^{r_1}\\
			&+(2\gamma_{r_2}-2)\frac{\gamma_{r_2}}{r_2}\lambda_2 \int_{\mathbb{R}^N}(I_\alpha*|v|^{r_2})|v|^{r_2})
			+\beta(\gamma_p+\gamma_q)\int_{\mathbb{R}^N}(I_\alpha*|u|^p)|v|^q\\
			&=\frac{2\gamma_{r_1}-\gamma_p-\gamma_q}{2\gamma_{r_1}-2}\beta(\gamma_p+\gamma_q)\int_{\mathbb{R}^N}(I_\alpha*|u|^p)|v|^q\\
			&\leq \frac{2\gamma_{r_1}-\gamma_p-\gamma_q}{2\gamma_{r_1}-2}(\gamma_p+\gamma_q)A_3
			(|\nabla u|_2^2+|\nabla v|_2^2)^{\frac{\gamma_p+\gamma_q}{2}}.
		\end{align*}
	On the other hand, by \eqref{GN} and \eqref{4.19 4.3.5}, we obtain
	\begin{align*}
			|\nabla u|_2^2+|\nabla v|_2^2&=\frac{\gamma_{r_1}}{r_1}\lambda_1 \int_{\mathbb{R}^N}(I_\alpha*|u|^{r_1})|u|^{r_1}
		+\frac{\gamma_{r_2}}{r_2}\lambda_2 \int_{\mathbb{R}^N}(I_\alpha*|v|^{r_2})|v|^{r_2}\\
		&\quad +\beta(\gamma_p+\gamma_q)\int_{\mathbb{R}^N}(I_\alpha*|u|^p)|v|^q\\
		&=\frac{2\gamma_{r_1}-\gamma_p-\gamma_q}{2-\gamma_p-\gamma_q}
		\frac{\gamma_{r_1}}{r_1}\lambda_1 \int_{\mathbb{R}^N}(I_\alpha*|u|^{r_1})|u|^{r_1}\\
		&\quad +\frac{2\gamma_{r_2}-\gamma_p-\gamma_q}{2-\gamma_p-\gamma_q}
		\frac{\gamma_{r_2}}{r_2}\lambda_2 \int_{\mathbb{R}^N}(I_\alpha*|v|^{r_2})|v|^{r_2}\\
		&\leq \frac{2\gamma_{r_1}-\gamma_p-\gamma_q}{2-\gamma_p-\gamma_q}2\gamma_{r_1}A_1
		(	|\nabla u|_2^2+|\nabla v|_2^2)^{\gamma_{r_1}}\\
		&\quad +\frac{2\gamma_{r_2}-\gamma_p-\gamma_q}{2-\gamma_p-\gamma_q}
		2\gamma_{r_2}A_2(|\nabla u|_2^2+|\nabla v|_2^2)^{\gamma_{r_2}}.
	\end{align*}
	Hence, we derive that
	\begin{align*}
		\frac{2\gamma_{r_1}-\gamma_p-\gamma_q}{2-\gamma_p-\gamma_q}2\gamma_{r_1}A_1
		(	|\nabla u|_2^2+|\nabla v|_2^2)^{\gamma_{r_1}-1}\geq \frac{1}{2},
	\end{align*}
	or
	\begin{align*}
		\frac{2\gamma_{r_2}-\gamma_p-\gamma_q}{2-\gamma_p-\gamma_q}
		2\gamma_{r_2}A_2(|\nabla u|_2^2+|\nabla v|_2^2)^{\gamma_{r_2}-1} \geq \frac{1}{2}.
	\end{align*}
		By combining the preceding inequalities, we get
	\[ \beta \geq  \beta_3  ~~  \text{or} ~~ \beta \geq  \beta_4, \]
	which contradicts $\beta<\beta_0$.
	
	By an argument analogous to the proof of Lemma \ref{Lemma4.1变分结构 4.3.1}                 , 
	we derive that the remaining conclusions hold; 
	hence, we omit it here.
\end{proof}
As the proof parallels the preceding case, we conclude the same existence result holds (Theorem \ref{Theorem1.13}).

\subsubsection{$p+q=\frac{2N+2\alpha+4}{N}<2r_1<2r_2<\frac{2N+2\alpha}{N-2}$}\leavevmode\par
Under the current assumption, we have $\gamma_p+\gamma_q=2<2\gamma_{r_1}<2\gamma_{r_2}<\frac{2N+2\alpha}{N-2}$ and
\begin{align*}
	h(s)=&\frac{1}{2}s^2-A_1s^{2\gamma_{r_1}}-A_2 s^{2\gamma_{r_2}}-A_3 s^{\gamma_p+\gamma_q}-\kappa \rho_1 \rho_2\\
	=&(\frac{1}{2}-A_3)s^2-A_1s^{2\gamma_{r_1}}-A_2 s^{2\gamma_{r_2}}-\kappa \rho_1 \rho_2,
\end{align*}
where $h(s)$ is given by \eqref{h(s)}. 
Assume that $\rho_1,\rho_2, \beta> 0$ satisfying
\begin{equation}
	\frac{1}{2}-A_3>0.
	\label{frac{1}{2}-A_3>0 4.3.6}
\end{equation}
A direct calculation yields
\begin{align*}
	h'(s)=&(1-2A_3)s-2\gamma_{r_1}A_1s^{2\gamma_{r_1}-1}-2\gamma_{r_2}A_2 s^{2\gamma_{r_2}-1}\\
	=&s\left[(1-2A_3)-2\gamma_{r_1}A_1s^{2\gamma_{r_1}-2}-2\gamma_{r_2}A_2 s^{2\gamma_{r_2}-2}\right],
\end{align*}
Set
\begin{align*}
	g(s):=2\gamma_{r_1}A_1s^{2\gamma_{r_1}-2}+2\gamma_{r_2}A_2 s^{2\gamma_{r_2}-2}.
\end{align*}
Direct computation shows that $\lim\limits_{s\to 0^+}g(s)=0$, $\lim\limits_{s\to +\infty}g(s)=+\infty$ and $g(s)$ is strictly increasing on $(0,+\infty)$.
It follows that there exists $s_0>0$ such that $g(s_0)=1-2A_3$ and 
$h(s)$ strictly increases on $(0,s_0)$ while decreases on $(s_0,+\infty)$. 

As the proof parallels the preceding case ($p+q =\frac{2N+2\alpha+4}{N}< 2r_1=2r_2<\frac{2N+2\alpha}{N-2}$), we conclude the same existence result holds (Theorem \ref{Theorem1.14}).

\subsubsection{$\frac{2N+2\alpha+4}{N}<p+q<2r_1<2r_2<\frac{2N+2\alpha}{N-2}$}\leavevmode\par
Under the current assumption, we have $2<\gamma_p+\gamma_q<2\gamma_{r_1}<2\gamma_{r_2}<\frac{2N+2\alpha}{N-2}$ and
\begin{align*}
	h(s)=\frac{1}{2}s^2-A_1s^{2\gamma_{r_1}}-A_2 s^{2\gamma_{r_2}}-A_3 s^{\gamma_p+\gamma_q}-\kappa \rho_1 \rho_2,
\end{align*}
where $h(s)$ is given by \eqref{h(s)}. 
A direct calculation yields
\begin{align*}
	h'(s)=&s-2\gamma_{r_1}A_1s^{2\gamma_{r_1}-1}-2\gamma_{r_2}A_2 s^{2\gamma_{r_2}-1}-(\gamma_p+\gamma_q)A_3 s^{\gamma_p+\gamma_q-1}\\
	=&s(1-2\gamma_{r_1}A_1s^{2\gamma_{r_1}-2}-2\gamma_{r_2}A_2 s^{2\gamma_{r_2}-2}-(\gamma_p+\gamma_q)A_3 s^{\gamma_p+\gamma_q-2}).
\end{align*}
Set
\begin{align*}
	g(s):=2\gamma_{r_1}A_1s^{2\gamma_{r_1}-2}+2\gamma_{r_2}A_2 s^{2\gamma_{r_2}-2}+(\gamma_p+\gamma_q)A_3 s^{\gamma_p+\gamma_q-2}.
\end{align*}
Direct computation shows that $\lim\limits_{s\to 0^+}g(s)=0$, $\lim\limits_{s\to +\infty}g(s)=+\infty$ and $g(s)$ is strictly increasing on $(0,+\infty)$.
It follows that there exists $s_0>0$ such that $g(s_0)=1$ and 
$h(s)$ strictly increases on $(0,s_0)$ while decreases on $(s_0,+\infty)$. 

As the proof parallels the preceding case ($\frac{2N+2\alpha+4}{N}<p+q<2r_1=2r_2<\frac{2N+2\alpha}{N-2}$), we conclude the same existence result holds (Theorem \ref{Theorem1.15} when $\frac{2N+2\alpha+4}{N}<p+q<2r_1<2r_2<\frac{2N+2\alpha}{N-2}$).

\subsubsection{$\frac{2N+2\alpha+4}{N}<p+q=2r_1<2r_2<\frac{2N+2\alpha}{N-2}$}\leavevmode\par
Under the current assumption, we have $2<\gamma_p+\gamma_q=2\gamma_{r_1}<2\gamma_{r_2}<\frac{2N+2\alpha}{N-2}$ and
\begin{align*}
	h(s)=&\frac{1}{2}s^2-A_1s^{2\gamma_{r_1}}-A_2 s^{2\gamma_{r_2}}-A_3 s^{\gamma_p+\gamma_q}-\kappa \rho_1 \rho_2\\
	=&\frac{1}{2}s^2-(A_1+A_3)s^{2\gamma_{r_1}}-A_2 s^{2\gamma_{r_2}}-\kappa \rho_1 \rho_2,
\end{align*}
where $h(s)$ is given by \eqref{h(s)}. 
A direct calculation yields
\begin{align*}
	h'(s)=&s-2\gamma_{r_1}(A_1+A_3)s^{2\gamma_{r_1}-1}-2\gamma_{r_2}A_2 s^{2\gamma_{r_2}-1}\\
	=&s(1-2\gamma_{r_1}(A_1+A_3)s^{2\gamma_{r_1}-2}-2\gamma_{r_2}A_2 s^{2\gamma_{r_2}-2}).
\end{align*}
Set
\begin{align*}
	g(s):=2\gamma_{r_1}(A_1+A_3)s^{2\gamma_{r_1}-2}+2\gamma_{r_2}A_2 s^{2\gamma_{r_2}-2}.
\end{align*}
Direct computation shows that $\lim\limits_{s\to 0^+}g(s)=0$, $\lim\limits_{s\to +\infty}g(s)=+\infty$ and $g(s)$ is strictly increasing on $(0,+\infty)$.
It follows that there exists $s_0>0$ such that $g(s_0)=1$ and 
$h(s)$ strictly increases on $(0,s_0)$ while decreases on $(s_0,+\infty)$. 

As the proof parallels the preceding case ($\frac{2N+2\alpha+4}{N}<p+q=2r_1=2r_2<\frac{2N+2\alpha}{N-2}$), we conclude the same existence result holds (Theorem \ref{Theorem1.15} when $\frac{2N+2\alpha+4}{N}<p+q=2r_1<2r_2<\frac{2N+2\alpha}{N-2}$).

	\section*{Acknowledgements}
The authors have no acknowledgements to declare.

\end{document}